\documentclass{article}
\usepackage{fullpage}

\usepackage[top=0.8in,bottom=1in,left=0.8in,right=0.8in]{geometry}
\usepackage{geometry}
\usepackage{graphicx}
\usepackage{subfigure}
\usepackage{float}
\usepackage{fancyhdr}
\usepackage{stmaryrd}
\usepackage{amsfonts}
\usepackage{amsthm, amssymb, amsmath}
\usepackage{mathrsfs}
\usepackage{diagbox}
\usepackage{verbatim}
\usepackage{lscape}
\usepackage{multirow}
\usepackage{algorithm}
\usepackage{algorithmic}
\usepackage{bm}
\usepackage[dvipsnames,table,xcdraw]{xcolor}
\usepackage{natbib}

\usepackage{colortbl}
\usepackage{booktabs}

\usepackage{enumitem}

\usepackage[colorlinks = true, pdfstartview = FitV, linkcolor = blue, citecolor = blue, urlcolor = blue]{hyperref}

\usepackage[capitalise]{cleveref}
\crefname{equation}{}{}
\crefname{figure}{Figure}{Figures}
\creflabelformat{equation}{\textup{(#2#1#3)}}
\crefname{assumption}{Assumption}{Assumptions}
\crefname{condition}{Condition}{Conditions}

\def\lmin{\lambda_{\min}}
\def\epsilong{\epsilon_g}

\def\epsilonh{\epsilon_H}

\def\deltagt{\delta_{g,t}}
\def\deltah{\delta_H}

\def\g{{\bf g}}
\def\H{{\bf H}}

\newcommand{\eye}{\bf{I}}

\def\a{{\bf a}}

\def\d{{\bf d}}

\def\f{{\bf f}}

\def\I{{\bf I}}

\def\p{{\bf p}}

\def\rr{{\bf r}}

\def\v{{\bf v}}

\def\x{{\bf x}}

\def\y{{\bf y}}

\def\z{{\bf z}}
\def\0{{\bf 0}}
\def\1{{\bf 1}}

\def\OM{{\mathcal O}}

\def\R{{\mathbb R}}

\def\sgn{\mathrm{sgn}}

\usepackage{tabularx,colortbl,xcolor}

\newtheorem{theorem}{Theorem}%[section]
\newtheorem{lemma}{Lemma}

\newtheorem{corollary}[theorem]{Corollary}
\newtheorem{definition}[theorem]{Definition}

\newtheorem{assumption}{Assumption}
\newtheorem{condition}{Condition}

\newcommand{\refone}[1]{{\leavevmode\color{black}#1}}
\newcommand{\reftwo}[1]{{\leavevmode\color{black}#1}}
\newcommand{\refboth}[1]{{\leavevmode\color{black}#1}}

\newcommand{\Abs}[1]{\left |#1\right|}
\newcommand{\norm}[1]{{\left\|#1\right\|}}

\newcommand{\bigO}{\mathcal{O}}

\def\eqref#1{(\ref{#1})}
\def\1{\bm{1}}

\def\rr{{\textnormal{r}}}

\def\vv{{\bm{v}}}

\def\mA{{\bm{A}}}

\DeclareMathAlphabet{\mathsfit}{\encodingdefault}{\sfdefault}{m}{sl}
\SetMathAlphabet{\mathsfit}{bold}{\encodingdefault}{\sfdefault}{bx}{n}

\newcommand\bbR{\ensuremath{\mathbb{R}}} % Real numbers
\newcommand{\flow}{f_{\text{\rm low}}}
\newcommand{\dtype}{d_{\text{\rm type}}}

\newcommand{\SOL}{\text{\sc SOL}}
\newcommand{\NC}{\text{\sc NC}}
\newcommand{\jsol}{j_{\text{\rm sol}}}
\newcommand{\csol}{c_{\text{\rm sol}}}
\newcommand{\jnc}{j_{\text{\rm nc}}}
\newcommand{\cnc}{c_{\text{\rm nc}}}
\newcommand{\csolb}{\bar{c}_{\text{\rm sol}}}
\newcommand{\cncb}{\bar{c}_{\text{\rm nc}}}
\usepackage{adjustbox}

\begin{document}

\title{Inexact Newton-CG Algorithms With Complexity Guarantees}

\author{
	Zhewei Yao
	\thanks{
    Department of Mathematics, University of California at Berkeley,
    Email: zheweiy@berkeley.edu
	}
	\and 
	Peng Xu
	\thanks{
	Amazon AWS AI,
	Email: pengx@amazon.com
	(Work done while at Institute for Computational and Mathematical Engineering, Stanford University.)
	}
	\and
	Fred Roosta
	\thanks{
		School of Mathematics and Physics, University of Queensland, Brisbane, Australia, and 
		International Computer Science Institute, Berkeley, USA,     
		Email: fred.roosta@uq.edu.au
	}
	\and 
	Stephen J. Wright
	\thanks{
	Computer Sciences Department, University of Wisconsin-Madison,
	Email: swright@cs.wisc.edu
	}
	\and
	Michael W. Mahoney
	\thanks{
	International Computer Science Institute and Department of Statistics, University of California at Berkeley,    
	Email: mmahoney@stat.berkeley.edu
	}
}

\maketitle

\begin{abstract}
    {\refone{We consider variants of a recently-developed Newton-CG algorithm for nonconvex problems \citep{royer2018newton} in which inexact estimates of the gradient and the Hessian information are used for various steps.} 
    Under certain conditions on the inexactness measures, we derive iteration complexity bounds for achieving $\epsilon$-approximate second-order optimality that match best-known lower bounds. 
    Our inexactness condition on the gradient is adaptive, allowing for crude accuracy in regions with large gradients.
    We describe two variants of our approach, one in which the step-size along the computed search direction is chosen adaptively and another in which the step-size is pre-defined.
    \refone{To obtain second-order optimality, our algorithms will make use of a negative curvature direction on some steps. 
    These directions can be obtained, with high-probability, using a certain randomized algorithm. 
    In this sense, all of our results hold with high-probability over the run of the algorithm.}
    We evaluate the performance of our proposed algorithms empirically on several machine learning models.}
    {Newton-CG, Non-Convex Optimization, Inexact Gradient, Inexact Hessian}
\end{abstract}

\section{Introduction}
We consider the following unconstrained optimization problem
\begin{equation} 
\label{eqn:basic_problem}
\min_{\x\in\R^d} f(\x),
\end{equation}
where $f:\bbR^d \to \bbR $ is a smooth but nonconvex function. 
At the heart of many machine learning and scientific computing applications lies the problem of finding an (approximate) minimizer of \cref{eqn:basic_problem}. 
Faced with modern ``big data'' problems, many classical optimization algorithms \citep{numopt,bertsekas1999nonlinear} are inefficient in terms of memory and/or computational overhead. 
Much recent research has focused on approximating various aspects of these algorithms. 
For example, efficient variants of first-order algorithms, such as the stochastic gradient method,  make use of inexact approximations of the gradient. 
The defining element of second-order algorithms is the  use of the curvature information from the Hessian matrix. 
In these methods, the main computational bottleneck lies with evaluating the Hessian, or at least being able to perform matrix-vector products involving the Hessian. 
Evaluation of the gradient may continue to be an unacceptably expensive operation in second-order algorithms too.
Hence, in adapting second-order algorithms to machine learning and scientific computing applications, we seek to {\em approximate} the computations involving the Hessian and the gradient, while preserving much of the convergence behavior of the exact underlying second-order algorithm. 

Second-order methods use curvature information to nonuniformly rescale the gradient in a way that often makes it a more  ``useful'' search direction, in the sense of providing a greater decrease in function value.
Second-order information also opens the possibility of convergence to points that satisfy second-order necessary conditions for optimality,  that is, $\x$ for which $\|\nabla f(\x)\|= 0$ and $\nabla^2 f(\x)\succeq \bm{0}$. 
For nonconvex machine learning problems, first-order stationary points include saddle points, which are undesirable for obtaining good generalization performance \citep{dauphin2014identifying,choromanska2015loss,saxe2013exact,lecun2012efficient}.

The canonical example of second-order methods is the classical Newton's method, which in its pure form is often written as
\begin{align*}
\x_{k+1} = \x_k + \alpha_k \d_k, \quad \mbox{where $\d_k =-\H_k^{-1} \g_k$,}
\end{align*}
where $\H_k = \nabla^2 f(\x_k)$ is the Hessian, $\g_k = \nabla f(\x_k)$ is the gradient, and $\alpha_k $ is some appropriate step-size, often chosen using an Armijo-type line-search \cite[Chapter~3]{numopt}. 
A more practical variant for large-scale problems is Newton-Conjugate-Gradient (Newton-CG), in which the linear system $\H_k \d_k = -\g_k$ is solved inexactly using the conjugate gradient (CG) algorithm \citep{steihaug1983conjugate}.
Such an approach requires access to the Hessian matrix only via matrix-vector products; it does not require $\H_k$ to be evaluated explicitly.

Recently, a new variant of the Newton-CG algorithm  was proposed in \cite{royer2018newton}  that can be applied to large-scale non-convex problems. 
This algorithm is equipped with certain safeguards and enhancements that allow worst-case complexity to be bounded in terms of the number of iterations and the total running time.
However, this approach relies on the exact evaluation of the gradient and on matrix-vector multiplication involving the exact Hessian at each iteration. 
Such operations can be prohibitively expensive in machine learning problems.
For example, when the underlying optimization problem has the finite-sum form
\begin{equation}
\label{eq:finte_sum_problem}
\min_{\x\in\R^d} f(\x) = \sum_{i=1}^n f_i(\x),
\end{equation}
exact computation of the Hessian/gradient can be costly when $ n \gg 1 $, requiring a complete pass through the training data set.
Our work here builds upon that of \cite{royer2018newton} but allows for \refboth{inexactness in computation of gradients and Hessians, while obtaining a similar complexity result to the earlier paper.}

\subsection{Related work}
Since deep learning became ubiquitous, first order methods such as gradient descent and its adaptive, stochastic variants \citep{kingma2014adam,duchi2011adaptive}, have become the most popular class of optimization algorithms in machine learning; see the recent textbooks \cite{beck2017first,lan2020first,lin2020accelerated,WriR21} for in-depth treatments.
These methods are easy to implement, and their per-iteration cost is low compared to second-order alternatives. 
Although classical theory for first-order methods guarantees convergence only to first-order optimal (stationary) points,  \cite{ge2015escaping,jin2017escape,levy2016power} argued that stochastic variants of certain first-order methods such as SGD have the potential of escaping saddle points and converging to second-order stationary points. 
The effectiveness of such methods usually requires painstaking fine-tuning of their (often many) hyperparameters, and the number of iterations they require to escape saddle regions can be large.

By contrast, second-order methods can make use of curvature information (via the Hessian) to escape saddle points efficiently and ultimately converge to second-order stationary points. This behavior is seen in trust-region methods \citep{conn2000trust,curtis2014trust,Cur19a}, cubic regularization \cite{nesterov2006cubic} and its adaptive variants (ARC) \citep{cartis2011adaptiveI,cartis2011adaptiveII}, as well as line-search based second-order methods \citep{royer2018complexity,royer2018newton}. 
Subsequent to \cite{cartis2011adaptiveI,cartis2011adaptiveII,cartis2012complexity}, which were among the first works to study Hessian approximations to ARC and trust region algorithms, respectively, \cite{xuNonconvexTheoretical2017} analyzed the optimal complexity of both trust region and cubic regularization, in which the Hessian matrix is approximated under milder conditions. 
Extension to gradient approximations was then studied in \cite{tripuraneni2017stochasticcubic,yao2018inexact}.
\refone{A novel take on inexact  gradient and dynamic Hessian accuracy is investigated in \cite{bellavia2021stochastic}.}
The analysis in \cite{gratton2018complexity,cartis2018global,blanchet2019convergence} relies on probabilistic models whose quality are ensured with a certain probability, but which allow for approximate evaluation of the objective function as well. 
\refone{Alternative approximations of the function and its derivative are considered in \cite{bellavia2019adaptive}.}

A notable difficulty of these methods concerns the solution of their respective subproblems, which can themselves be nontrivial nonconvex optimization problems. 
Some exceptions are \cite{royer2018newton,liu2019stability,roosta2018newton}, whose fundamental operations are linear algebra computations, which are much better understood. While \cite{liu2019stability,roosta2018newton} are limited in their scope to invex problems \citep{mishra2008invexity}, the method in \cite{royer2018newton} can be applied to more general non-convex settings.
In fact, \cite{royer2018newton} enhances the classical Newton-CG approach with safeguards to detect negative curvature in the Hessian, during the solution of the Newton equations to obtain the step $\d_k$. 
Negative curvature directions can subsequently be exploited by the algorithm to make significant progress in reducing the objective. 
Moreover, \cite{royer2018newton} gives complexity guarantees that have been shown to be optimal in certain settings.
(Henceforth, we use the term ``Newton-CG'' to refer specifically to the algorithm in \cite{royer2018newton}.)

\subsection{Contribution}

\iffalse
{\bf PREVIOUS TEXT:}
\refone{We describe two variants of the Newton-CG algorithm in \cite{royer2018newton} in which, to reduce overall computational costs, approximations of gradient and Hessian are employed at various steps. In the first variant (\cref{alg:inexact_withlinesearch}), we consider a line-search based algorithm where we entirely replace the exact Hessian with its approximations. To compute the Newton direction (Procedure \ref{alg:capped_cg}), we also employ approximation of the gradient. However, to subsequently obtain the step-size, we resort to using the exact function and gradient information. This is of course not ideal since in many problems, evaluating the gradient (or the objective function) exactly can be prohibitive. 
To partially remedy this situation, we propose a second variant (\cref{alg:fixed_stepsize}) which, by employing constant step-sizes, obviates the need for function evaluations and allows for an approximation to the gradients to be used throughout. 
The main drawback of this variant is that the fixed step-size depends on bounds on problem-dependent quantities. 
While these are available in several problems of interest in machine learning and statistics (see \cref{tab:L_H_bound,tab:K_H_K_g_bound}), they may be hard to estimate for other practical problems. 
Moreover, the step-sizes obtained from these bounds tend to be conservative, a situation that arises often in fixed-step optimization methods.}
\fi

\refone{We describe two new variants of the Newton-CG algorithm of \cite{royer2018newton} in which, to reduce overall computational costs, approximations of gradient and Hessian are employed. 
The first variant (\cref{alg:inexact_withlinesearch}) is a line-search method in which only approximate gradient and Hessian information is needed at each step, but it resorts to the use of exact function values in performing a backtracking line search at each iteration.  
This requirement is not ideal, since exact evaluation of the objective function can be prohibitive. 
To partially remedy this situation, we propose a second variant (\cref{alg:fixed_stepsize}) which, by employing constant step-sizes, obviates the need for exact evaluations of functions, gradients, or Hessians.
The main drawback of this variant is that the fixed step-size depends on bounds on problem-dependent quantities. 
While these are available in several problems of interest in machine learning and statistics (see \cref{tab:L_H_bound,tab:K_H_K_g_bound}), they may be hard to estimate for other practical problems. 
Moreover, the step-sizes obtained from these bounds tend to be conservative, a situation that arises often in fixed-step optimization methods.}

\refboth{For both these algorithms, we show that the convergence and complexity properties of the original exact algorithm from \cite{royer2018newton} are largely retained.} 
Specifically, to achieve ($\epsilon$, $\sqrt{\epsilon}$)-optimality (see Definition~\ref{def:optimality} below) under Condition \ref{cond:opt_epsilong_epsilonh} on gradient and Hessian approximations (see below, in Section \ref{sec:opt}), we show the following.
\begin{itemize}
	\item Inexact Newton-CG with backtracking line search (\cref{alg:inexact_withlinesearch}), achieves the optimal iteration
	complexity of $\bigO(\epsilon^{-3/2})$; see Section \ref{sec:opt}. 
	\item Inexact Newton-CG in which a predefined step size replaces the backtracking line searches (\cref{alg:fixed_stepsize}) achieves the same optimal iteration
	complexity of $\bigO(\epsilon^{-3/2})$; see Section~\ref{sec:fixed_step_size}. 
	\item We obtain estimates of oracle complexity in terms of $\epsilon$ for both variants.
	\item The accuracy required in our gradient approximation changes adaptively with the current gradient size.
	One consequence of this feature is to allow cruder gradient approximations in the regions with larger gradients, translating to a more efficient algorithm overall.
	\item We empirically illustrate the advantages of our methods on several real datasets; see Section \ref{sec:numerical_experiments}. 
\end{itemize}

We note that \cref{alg:inexact_withlinesearch} \refboth{may not be computationally feasible as written, because 
the backtracking line searches require repeated (exact) evaluation of $f$.
This requirement} may not be practical in situations in which \refboth{exact evaluations of $ f$} are impractical.
By contrast, \cref{alg:fixed_stepsize} does not assume such knowledge and can be implemented strictly as written, given knowledge of the appropriate Lipschitz constant. The steplengths used in Algorithm~\ref{alg:fixed_stepsize} are, however, quite conservative, and better computational results will almost certainly be obtained with Algorithm~\ref{alg:inexact_withlinesearch}, modified to use \refboth{ approximations to $f(\x)$}; see the numerical examples in \cref{sec:numerical_experiments}.

\section{Algorithms and analysis}
\label{sec:main_result}

We describe our algorithms and present our main theoretical results in this section.
We start with background (Section~\ref{sec:notation_and_assumption}) and important technical ingredients (Section~\ref{sec:main_ingredients}),  and then we proceed to our two main algorithms (Section~\ref{sec:opt} and Section~\ref{sec:fixed_step_size}).

\subsection{Notation, definitions, and assumptions}
\label{sec:notation_and_assumption}
Throughout this paper, scalar constants are denoted by regular lower-case and upper-case letters, e.g., $c$ and $K$. 
We use bold lowercase and blackboard bold uppercase letters to denote vectors and matrices, e.g., $\a$ and $\mA$, respectively. 
The transpose of a real vector $\a$ is denoted by $ \a^{T} $.
For a vector $\a$, and a matrix $\mA$, $\|\a\|$ and $\|\mA\|$ denote the vector $\ell_{2}$ norm and the matrix spectral norm, respectively.
Subscripts (as in $\a_{t}$) denote iteration counters. 
The smallest eigenvalue of a symmetric matrix $\mA$ is denoted by $\lambda_{\min}(\mA)$. 
For any $ \x,\y \in \R^{d}$, $ \left[\x, \y\right] $ denotes the line segment between $\x$ and $\y$, i.e., $\left[\x, \y\right] = \left\{\z \mid \z = \x + \tau (\y - \x), \; 0 \leq \tau \leq 1 \right\}$.

We are interested in expressing certain bounds in terms of their dependence on the small positive convergence tolerance $\epsilon$, especially on certain negative powers of this quantity, ignoring the dependence on all other quantities in the problem, such as dimension, Lipschitz constants, etc. For example, we use $\OM(\epsilon^{-1})$ to denote a bound that depends linearly on $\epsilon^{-1}$ and $\tilde\OM(\epsilon^{-1})$ for linear dependence on $\epsilon^{-1} |\log\epsilon|$.

For nonconvex problems, the determination of near-optimality can be much more complicated than for convex problems; see the examples of~\cite{murty1987some,hillar2013most}. 
In this paper, as in earlier works (see for example \cite{royer2018newton}), we make use of approximate second-order optimality, defined as follows.
\begin{definition}[$(\epsilong,\epsilonh)$-optimality]\label{def:optimality}
	Given $0<\epsilong,\epsilonh<1$, $\x$ is an $(\epsilong,\epsilonh)$-optimal solution of \eqref{eqn:basic_problem},~if
% 	\footnote{Throughout this paper, $\|\cdot\|$ is the $\ell_2$ norm by default, and $\lmin(\cdot)$ is the minimum eigenvalue.}
	\begin{align}\label{cond:stop_cr}
	\|\nabla f(\x)\| \leq \epsilong \quad \text{and} \quad \lmin( \nabla^2 f(\x) )  \geq -\epsilonh.
	\end{align}
\end{definition}

\begin{assumption}\label{ass:lip}
The smooth nonconvex function $f$ is bounded below by the finite value $\flow$. \reftwo{It also has compact sub-level sets, i.e., the set $ \mathcal{L}(\x_{0}) = \left\{\x \mid f(\x) \leq f(\x_{0})\right\} $ is compact.} 
Moreover, on an open set $\mathcal{B} \subset \R^n$ containing all line segments $[\x_k,\x_k+\d_k]$ for iterates $\x_k$ and search directions $\d_k$ generated by our algorithms, the objective function has Lipschitz continuous gradient and Hessian, that is, there are positive constants $0 < L_g < \infty$ and $0 < L_H < \infty$ such that for any $ \x,\y \in \mathcal{B} $, we have
	\[
	\norm{\nabla f(\x) - \nabla f(\y)} \le L_g \norm{\x - \y} \quad \text{and} \quad
	\norm{\nabla^2 f(\x) - \nabla^2 f(\y)} \le L_H \norm{\x - \y} .
	\]
\end{assumption} 
\refone{Although \cref{ass:lip} is typical in the optimization literature, it nonetheless implies a somewhat strong smoothness assumptions on the function. Some related works on various Newton-type methods, e.g., \cite{bellavia2019adaptive,bellavia2021stochastic}, obtain second-order complexity guarantees that require only Lipschitz continuity of the Hessian. It would be interesting to investigate whether our analysis can be modified to allow for such relaxations. We leave such investigations for future work.}

Consequences of Lipschitz continuity of the Hessian, which we will use in later results, include the following bounds for any $\x,\y \in \mathcal{B}$:
\begin{subequations}
\begin{align}
    \label{eq:lipH}
& \norm{ \nabla f(\x) - \nabla f(\y) - \nabla^2 f(\y)(\x-\y)} \le \frac{L_H}{2} \norm{\x-\y}^2 \\
\label{eq:lipH2}
& f(\x) \le f(\y) + \nabla f(\y)^T(\x-\y) + \frac12 (\x-\y)^T\nabla^2 f(\y)(\x-\y) + \frac{L_H}{6} \| \x-\y\|^3.
\end{align}
\end{subequations}
\refone{An interesting avenue for future research is to try to replace these Lipschitz continuity conditions with milder variants in which the gradient and/or Hessian are required to maintain Lipschitz continuity only along a given set of directions, e.g., the piecewise linear path generated by the iterates such as the corresponding assumption in \cite{xuNonconvexTheoretical2017}. 
Our current proof techniques do not allow for such relaxations, but we will look into possibility in future work.}

For our inexact Newton-CG algorithms, we also require that the approximate gradient and Hessian satisfy the following conditions, for prescribed positive values $ \deltagt $ and $ \deltah $.
\begin{condition}\label{cond:appr_gh}
	For given $\deltagt$ and $\deltah$, we say that the approximate gradient $\g_t$ and Hessian $\H_t$ at iteration $t$ are $ \deltagt$-accurate and  $\deltah$-accurate if 
	\[
	\|\g_t - \nabla f(\x_t)\| \leq \deltagt \quad \text{and} \quad\|\H_t - \nabla^2 f(\x_t)\| \leq \deltah, 
	\]
	respectively.
\end{condition}

Under these assumptions and conditions, it is easy to show that there exist constants  $U_g$ and $U_H$ such that the following are satisfied for all iterates $\x_t$ in the set defined in Assumption~\ref{ass:lip}:
\begin{equation} \label{eq:bounds}
\|\g_t\| \leq U_g\quad \text{and} \quad\|\H_t\| \leq U_H.
\end{equation}

\subsection{Key ingredients of the Newton-CG method}
\label{sec:main_ingredients}

We present the two major components from \cite{royer2018newton} that are also used in our inexact variant of the Newton-CG algorithm.
The first ingredient, Procedure~\ref{alg:capped_cg} (referred to in some places as ``Capped CG''), is a version of the conjugate gradient~\citep{shewchuk1994introduction} algorithm that is used to solve a damped Newton system of the form $\bar \H \d = -\g$, where $\bar\H = \H + 2\epsilon \eye$ for some positive parameter $\epsilon$.
Procedure~\ref{alg:capped_cg} is modified to detect indefiniteness in the matrix $\H$ and, when this occurs, to return a direction along which the curvature of $\H$ is at most $-\epsilon$.
The second ingredient, 
Procedure~\ref{alg:minimum_eigenvalue} (referred to as the ``Minimum Eigenvalue Oracle'' or ``MEO''), checks whether a direction of  negative curvature (less than $-\epsilon$ for a given positive argument $\epsilon$) exists for the given matrix $\H$. 
We now discuss each of these procedures in more detail.

\begin{algorithm}[!ht]
\footnotesize
\floatname{algorithm}{Procedure}
    \caption{Capped Conjugate Gradient}
    \label{alg:capped_cg}
    \begin{algorithmic}[1]
        \STATE {\bf Inputs:} Symmetric Matrix $\H\in \R^{d\times d}$, vector $\g \ne 0$; damping parameter $\epsilon \in (0,1)$; ; desired accuracy $\zeta \in (0, 1)$;
        \STATE {\bf Optional input:} positive scale $M$ (set to 0 if not provided) 
        \STATE {\bf Outputs:} $\dtype$, $\d$
        \STATE {\bf Secondary Output:} $M$, $\kappa$, $\tilde \zeta$, $\tau$, $T$
        \STATE Set
        \[
        \bar\H:=\H + 2\epsilon, \quad \kappa:=\frac{M+2\epsilon}{\epsilon}, \quad \tilde\zeta:=\frac{\zeta}{3\kappa}, 
        \quad T := \frac{4\kappa^4}{(1 - \sqrt{1-\tau})^2}, \quad \tau :=\frac{1}{\sqrt\kappa + 1};
        \] 
        \STATE $\y_0 \gets 0, \rr_0 \gets \g, \p_0 \gets -\g, j \gets 0$
        \IF {$\p_0^T\bar\H \p_0 < \epsilon \norm{\p_0}^2$}
        \STATE Set $\d = p_0$ and terminate with $\dtype=\NC$;
        \ELSIF {$\|\H\p_0\| > M\|\p_0\|$}
        \STATE $M \gets \|\H\p_0\| / \|\p_0\|$ and update $\kappa$, $\tilde \zeta$, $\tau$, $T$;
        \ENDIF
        \WHILE {TRUE}
        \STATE $\alpha_j \gets \rr_j^T\rr_j/\p_j^T\bar\H \p_j$;~~(Traditional CG Begins)
        \STATE $\y_{j+1} \gets \y_j + \alpha_j \p_i$;
        \STATE $\rr_{j+1} \gets \rr_j + \alpha_j \bar\H \p_j$;
        \STATE $\beta_{j+1} \gets \rr_{j+1}^T \rr_{j+1} /\rr_j^T\rr_j$;
        \STATE $\p_{j+1} \gets  -\rr_{j+1} + \beta_{j+1} \p_j$;~~(Traditional CG Ends)
        \STATE $j\gets j+1$;
        \IF {$\max (\|\H\p_j\| / \|\p_j\|, \|\H\y_j\| / \|\y_j\|, \|\H\rr_j\| / \|\rr_j\|)>M$}
        \STATE $M \gets \max (\|\H\p_j\| / \|\p_j\|, \|\H\y_j\| / \|\y_j\|, \|\H\rr_j\| / \|\rr_j\|) $ and update $\kappa$, $\tilde \zeta$, $\tau$, $T$;
        \ENDIF
        \IF {$\y_j^T\bar\H \y_j \le \epsilon \norm{\y_j}^2$} \label{line:nca}
        \STATE Set $\d\gets \y_j$ and terminate with $\dtype=\NC$;
        \ELSIF {$\norm{\rr_j} \le\hat\zeta\norm{\rr_0}$}
        \STATE Set $\d\gets \y_j$ and terminate with $\dtype=\SOL$;
        \ELSIF {$\p_j^T\bar\H \p_j \le \epsilon \norm{\p_j}^2$} \label{line:ncb}
        \STATE Set $\d\gets \p_j$ and terminate with $\dtype=\NC$;
        \ELSIF {$\norm{\rr_j} \ge \sqrt T (1 - \tau)^{j/2} \norm{\rr_0}$} \label{line:ncc}
            \STATE Compute $\alpha_j, \p_{j+1}$ as in the main loop above; \label{line:nc1}
        \STATE Find $i\in\{0,\cdots,j-1\}$ such that \label{line:nc2}
        \begin{equation}\label{eq:capped_cg_accum}
        \frac{(\y_{j+1} - \y_i)^T\bar\H (\y_{j+1} -\y_i)}{\norm{\y_{j+1}-\y_i}^2} \le \epsilon;    
        \end{equation}
        \STATE Set $\d \gets \y_{j+1} - \y_i$ and terminate with $\dtype=\NC$; \label{line:nc3}
        \ENDIF
        \ENDWHILE
        \STATE {\bf Return:} $\d$
    \end{algorithmic}
\end{algorithm}
\paragraph{Procedure~\ref{alg:capped_cg} (Capped-CG).}
The well-known classical CG algorithm \citep{shewchuk1994introduction} is used to solve linear systems involving positive definite matrices.
However, this positive-definite requirement is often violated during the iterations for non-convex optimization due to the indefiniteness of Hessians encountered at some iterates. 
Capped-CG, proposed by \cite{royer2018newton} and presented in Procedure~\ref{alg:capped_cg} for completeness, is an original way to leverage and detect such negative curvature directions, when they are encountered during CG iterations.

Lines 13-17 in Procedure~\ref{alg:capped_cg} contain the standard CG operations.
When $\H \succeq -\epsilon \I$, the tests in lines 22, 26, and 28 that indicate negative curvature will not be activated, and Capped-CG will return an approximate solution $\d \approx -\bar\H^{-1}\g$.  
However, when $\H \not\succeq -\epsilon \I$, Capped-CG will identify and return a direction of ``sufficient negative curvature'' --- a direction $d$  satisfying $\d^T\H\d \leq -\epsilon\|\d\|^2$.
Such a negative curvature direction is obtained under two circumstances. 
First, when the intermediate step (either $\y_j$ or $\p_j$) satisfies the negative curvature condition, that is, $\d^T\bar\H \d \leq -\epsilon \|\d\|^2$ (Lines \ref{line:nca} and \ref{line:ncb}), Procedure~\ref{alg:capped_cg} will be terminated and the intermediate step will be returned. 
Second, when the residual, $\rr_j$, decays at a slower rate than anticipated by standard CG analysis (Line \ref{line:ncc}), a negative curvature direction can be recovered by the procedure of Lines \ref{line:nc1}, \ref{line:nc2}, and \ref{line:nc3}.
Note that Procedure~\ref{alg:capped_cg} can be called with an optional input $M$, which is an upper bound on $\|\H\|$. 
However, even without a priori knowledge of this upper bound, M can be updated so that at any point in the execution of the procedure, M is an upper bound on the maximum curvature of $ \H $ revealed to that point. 
Other parameters ($\kappa$, $\tilde \zeta$, $\tau$, $T$) are also updated whenever the value of M changes. 
It is not hard to see that  $M$ is bounded by $U_\H$ throughout the execution of Procedure~\ref{alg:capped_cg}, provided that if an initial value of $M$ is supplied to this procedure, this value satisfies  $M \le U_\H$.

Lemma~\ref{lemma:capped_cg_iter_bounded} gives a bound on the number of iterations performed by Procedure~\ref{alg:capped_cg}. 
\begin{lemma}[{\citet[Lemma 1]{royer2018newton}}] 
\label{lemma:capped_cg_iter_bounded}
The number of iterations of Procedure~\ref{alg:capped_cg} is bounded by 
\[
    \min \, \left\{ d, J(M, \epsilon, \zeta) \right\},
\]
where $J = J(M,\epsilon, \zeta)$ is the smallest integer such that $\sqrt T (1-\tau)^{J/2} \leq \hat \zeta$. The number of matrix-vector products required is bounded by $2\min\{d, J(M, \epsilon, \zeta)\}+1$, unless all iterates $\y_i,~i=1,2,\ldots$ are stored, in which case it is $\min\{d, J(M, \epsilon, \zeta)\} + 1$.
For the upper bound of $J(M, \epsilon, \zeta)$, we have
\begin{equation}
    J(M, \epsilon, \zeta) \leq \min \, \left\{d, \tilde\OM(\epsilon^{-1/2})\right\}.
\end{equation}
\end{lemma}

When the slow decrease in residual is detected (Line 21), a direction of negative curvature for $\H$ can be extracted from the previous intermediate solutions, as the following result describes.
\begin{lemma}[{\citet[Theorem 2]{royer2018newton}}]
\label{lemma:capped_cg_last_condition_gaurantee}
Suppose that the loop of Procedure~\ref{alg:capped_cg} terminates with $j=\hat J$, where
\[
    \hat J \in \{1,2,\ldots, \min\{n, J(M, \epsilon, \zeta)\}\}
\]
satisfies
\[
    \|r_{\hat J}\| > \max\{\hat\zeta, \sqrt T (1-\tau)^{\hat J/2}\}\|r_0\|.
\]
Suppose further that $y_{\hat J}^T\bar \H y_{\hat J} \geq \epsilon \|y_{\hat J}\|^2$, so that $y_{\hat J + 1}$ is computed. 
Then we have 
\[
    \frac{(y_{\hat J +1} - y_i)^T\bar \H (y_{\hat J +1} - y_i)}{\|y_{\hat J +1} - y_i\|^2} < \epsilon,~~~~for~~some~~i\in\{0,\ldots,\hat J -1\}.
\]
\end{lemma}

Note that $d^T \bar\H d \le \epsilon \|d\|^2 \Longleftrightarrow d^T \H d \le -\epsilon \|d\|^2$.

Procedure~\ref{alg:capped_cg} is invoked by the Newton-CG procedure, Algorithm~\ref{alg:inexact_withlinesearch} (described in \cref{sec:opt}), when the current iterate $\x_k$ has $\| \g_k \| \ge \epsilong>0$. 
Procedure~\ref{alg:capped_cg} can either return the approximate Newton direction or a negative curvature one. 
After describing how this output vector is modified by Algorithm~\ref{alg:inexact_withlinesearch}, in the next section, we state a result  (Lemma~\ref{lemma:d_from_cappedcg}) about the properties of the resulting step.

In the case of $\|\g_k \|<\epsilong$,  Algorithm~\ref{alg:inexact_withlinesearch} calls Procedure~\ref{alg:minimum_eigenvalue} to explicitly seek a direction of sufficient negative curvature. 
We describe this procedure next.

\begin{algorithm}[H]
\floatname{algorithm}{Procedure}
\caption{Minimum Eigenvalue Oracle}
\label{alg:minimum_eigenvalue}
\begin{algorithmic}[1]
\STATE {\bf Inputs:} Symmetric matrix $\H\in\R^{d\times d}$, scalar $M \ge \lambda_{\max}(\H)$ and $\epsilon >0$;
\STATE Set $\delta \in [0,1)$;
\STATE {\bf Outputs:} Estimate $\lambda$ of $\lambda_{\min}(\H)$ such that $\lambda  \le - \epsilon/2$ and vector $\v$ with $\norm{\v} =1$ such that $\v^T\H\v = \lambda$ OR certificate that $\lambda_{\min}(\H) \ge -\epsilon$. The probability that the certificate is issued but $\lambda_{\min}(\H) < -\epsilon$ is at most $\delta$.
\end{algorithmic}
\end{algorithm}

\paragraph{Procedure~\ref{alg:minimum_eigenvalue} (Minimum Eigenvalue Oracle).}
This procedure searches for a direction spanned by the negative spectrum of a given symmetric matrix or, alternately, verifies that the matrix is (almost) positive definite. 
Specifically, for a given $\epsilon>0$, Procedure~\ref{alg:minimum_eigenvalue} finds a negative curvature direction $\v$ of $\H_k$ such that $\v^T\H\v\leq -\epsilon\|\v\|^2/2$, or else certifies that  $\H \succeq -\epsilon \I$. 
The probability that the certificate is issued but $\lambda_{\min} (\H) < -\epsilon$ is bounded above by some (small) specified value $\delta$.
As indicated in \cite{royer2018newton}, this minimum eigenvalue oracle can be implemented using the Lanczos process or the classical CG algorithm.  (In this paper, we choose the former.) 
Both of these approaches have the same complexity, given in the following result.
\begin{lemma}[{\citet[Lemma 2]{royer2018newton}}] 
\label{lemma:fail_min_eigen_oracle}
Suppose that the Lanczos method is used to estimate the smallest eigenvalue of $\H$ starting from a random vector drawn from the uniform distribution on the unit sphere, where $\|\H \| \le M$. 
For any $\delta \in (0,1)$, this approach finds the smallest eigenvalue of $H$ to an absolute precision of $\epsilon/2$, together with a corresponding direction $\v$, in  at most 
\begin{equation}\label{eqn:iter_neg}
    \min\, \left\{d, 1+ \left\lceil \frac{\ln(2.75d/\delta^2)}{2}\sqrt{\frac{M}{\epsilon}} \right\rceil \right\} \quad \mbox{iterations,}
\end{equation}
with probability at least $1-\delta$. Each iteration requires evaluation of a matrix-vector product involving $\H$.
\end{lemma}

\subsection{Inexact Newton-CG algorithm with line search} 
\label{sec:opt}

\begin{algorithm}[!ht]
\caption{Inexact Damped Newton-CG with Line Search}
\footnotesize
\label{alg:inexact_withlinesearch}
\begin{algorithmic}[1]
\STATE {\bf Inputs:} $\epsilon_g, \epsilon_H >0$; backtracking parameter $\theta \in(0,1)$; sufficient decrease parameter $\eta>0$; starting point $\x_0$; upper bound  on Hessian  norm $U_H>0$; accuracy  parameter $\zeta \in (0, \min\{1,U_H\})$;
\FOR {$k=0,1,2,\cdots$}
\IF {$\norm{\g_k}  \ge \epsilon_g$} 
\STATE Call Procedure~\ref{alg:capped_cg} with $\H = \H_k, M=U_H, \epsilon = \epsilon_H, \g = \g_k$ and accuracy parameter $\zeta$ to obtain $\d$ and $\dtype$;
\IF {$\dtype == \NC$}
\STATE \refboth{$\d_k \gets -\sgn(\d^T \g_k)\frac{|\d^T\H_k\d|}{\norm{\d}^2} \frac{\d}{\norm{\d}}$} and go to {\bf Line-Search} ;
\ELSE
\STATE $\d_k \gets \d$;
\IF{$\norm{\d_k} \le \epsilong/\epsilonh$} \label{line:threshold}
\STATE Call Procedure~\ref{alg:minimum_eigenvalue} with $\H = \H_k, M= U_H, \epsilon= \epsilon_H$ to obtain $\v$ (with $\norm{\v}=1$ and $\v^T \H_k \v \le -\epsilon_H/2$) or a certificate that $\lambda_{\min}(\H_k) \ge -\epsilonh$;
\IF {Procedure~\ref{alg:minimum_eigenvalue} certifies that $\lambda_{\min}(\H_k) \ge -\epsilonh$}
\STATE Terminate and return $\x_k+\d_k$; \label{line:term1}
\ELSE
\STATE \refboth{$\d_k \gets -\left( \sgn(\v^T\g_k) |\v^T\H_k\v| \right) \v$}, $\dtype \gets \NC$, and go to {\bf Line-Search}; \label{line:newdk}
\ENDIF
\ELSE
\STATE Go to {\bf Line-Search}; 
\ENDIF \label{line:threshold_end}
\ENDIF
\ELSE
\STATE $\dtype \gets \NC$;
\STATE Call Procedure \ref{alg:minimum_eigenvalue} with $\H = \H_k, M= U_H, \epsilon= \epsilon_H$ to obtain $\v$ with $\norm{\v}=1$ and $\v^T \H_k \v \le -\epsilon_H/2$ or a certificate that $\lambda_{\min}(\H_k) \ge -\epsilonh$;
\IF {Procedure~\ref{alg:minimum_eigenvalue} certifies that $\lambda_{\min}(\H_k) \ge -\epsilonh$}
\STATE Terminate and return $\x_k$; \label{line:term2}
\ELSE
\STATE \refboth{$\d_k \gets -\sgn(\v^T\g_k) |\v^T\H_k\v| \v$} and go to {\bf Line-Search};
\ENDIF
\ENDIF
\STATE {\bf Line-Search:} 
\refboth{\IF {$\dtype == \SOL$}
\STATE Set $\alpha_k \gets \theta^{j_k}$, where $j_k$ is the smallest  nonnegative integer such that 
\begin{equation} \label{eq:suffdecr}
f(\x_k + \alpha_k \d_k) < f(\x_k) - \frac{\eta}{6}|\alpha_k|^3 \norm{\d_k}^3; 
\end{equation}
\ELSE 
\STATE Set $\alpha_k$ to be the first element of the sequence $1, -1, \theta, -\theta, \theta^2, -\theta^2, \theta^3, -\theta^3, \dotsc$ for which  \eqref{eq:suffdecr} holds;
\ENDIF
}
\STATE $\x_{k+1} \gets \x_k + \alpha_k \d_k$;
\ENDFOR
\end{algorithmic}
\end{algorithm}

\cref{alg:inexact_withlinesearch} shows our inexact damped Newton-CG algorithm, which calls Procedures~\ref{alg:capped_cg} and~\ref{alg:minimum_eigenvalue}. In this section, we establish worst case iteration complexity to achieve $(\epsilong,\epsilonh)$-optimality  according to \cref{def:optimality}. 
Under mild conditions on the approximate gradient and Hessian, the complexity estimate is the same as for the exact Newton-CG algorithm described in  \cite{royer2018newton}.

\refone{For \cref{alg:inexact_withlinesearch}, approximations of the Hessian and gradient can be used throughout. However, to obtain the step-size $ \alpha_k $, \cref{alg:inexact_withlinesearch} requires exact evaluation of the function. 
We avoid the need for these exact evaluations in  the fixed-step variant, \cref{alg:fixed_stepsize}, to be studied in \cref{sec:fixed_step_size}.}

Apart from the use of approximate Hessian and gradient, Lines  \ref{line:threshold}-\ref{line:threshold_end} constitute a notable difference between our algorithm and the exact counterpart of \cite{royer2018newton}, in which our method calls  Procedure~\ref{alg:minimum_eigenvalue} to obtain a direction of sufficient negative curvature when the direction $\d_k$ derived  from Procedure~\ref{alg:capped_cg} is small; specifically, $\|\d_k\|\leq \epsilong/\epsilonh$.
If such a direction is found, we perform a backtracking line search along with it.
Otherwise, if Procedure~\ref{alg:minimum_eigenvalue} certifies that no direction of sufficient negative curvature exists, we terminate and return the point $\x_k+\d_k$, which already satisfies the second-order optimality condition.
In theory, this modification is critical to obtaining the optimal worst-case complexity. 
\reftwo{In practice, however, we have observed that performing line-search with such $ \d_k $, despite the fact that $\|\d_k\|\leq \epsilong/\epsilonh$, results in acceptable progress in reducing the function. In other words, we believe that Lines 9-16 of \cref{alg:inexact_withlinesearch,alg:fixed_stepsize} serve a mainly theoretical purpose and can be safely omitted in practical implementations.}

\refboth{Another notable difference with previous versions of this general approach is the use of a ``bidirectional'' line search when $\d_k$ is a negative curvature direction. 
We do backtracking along both positive and negative directions, $\d_k$ and $-\d_k$, because we are unable to determine with certainty the sign of $\d_k^T \nabla f(\x_k)$, since we have access only to the approximation $\g_k$ of $\nabla f(x_k)$. 
This additional algorithmic feature causes only modest changes to the analysis of the function decrease along negative curvature directions, as we point out in the appropriate results below.}

We begin our complexity analysis with a result that summarizes important properties of the direction $\d_k$ that is derived from the capped CG algorithm, Procedure~\ref{alg:capped_cg}.
(The proof is identical to that of the cited result \cite[Lemma~3]{royer2018newton}, except that we use approximate values of the Hessian and gradient of $f$ here.)
\begin{lemma}[{\citet[Lemma~3]{royer2018newton}}] 
\label{lemma:d_from_cappedcg}
Suppose that Assumption~\ref{ass:lip} is satisfied.
Suppose that Procedure~\ref{alg:capped_cg} is invoked at an iterate $\x_k$ of Algorithm~\ref{alg:inexact_withlinesearch} (so that $\|\g_k\| \ge \epsilong>0$) with inputs $\H=\H_k$, $\g=\g_k$, $\epsilon=\epsilonh$, and $\zeta$. 
Suppose that $\d_k$ in Algorithm~\ref{alg:inexact_withlinesearch} is obtained from the output vector $\d$ of Procedure~\ref{alg:capped_cg}, after possible scaling and change of sign.
Then one of the two following statements holds.
\begin{enumerate}
\item $\dtype=\SOL$ and $\d_k=\d$ satisfies
\begin{subequations}
\begin{align}\label{eqn:pos}
\d_k^T \H_k \d_k \geq -\epsilonh\|\d_k\|^2,
\end{align}
\begin{align}\label{eqn:dk_sol_upper}
\|\d_k\| \leq 1.1 \epsilonh^{-1} \|\g_k\|,
\end{align}
\begin{align}\label{eqn:dk_sol_rk}
\|\hat \rr_k\|  \leq \frac12 \epsilonh\zeta\|\d_k\|,
\end{align}
\end{subequations}
where 
\begin{align}\label{eq:r_k}
    \hat \rr_k = (\H_k + 2\epsilonh\I)\d_k + \g_k.
\end{align}
\item $\dtype=\NC$ and $\d_k$ satisfies 
\[
\d_k = -\sgn(\d^T\g_k)\frac{|\d^T\H_k\d|}{\norm{\d}^2} \frac{\d}{\norm{\d}},
\]
and $\d_k$ satisfies
\begin{equation}\label{eqn:dk_nc}
\frac{\d_k^T\H_k\d_k}{\| \d_k \|^2} = -\| \d_k \| \le -\epsilonh.
\end{equation}
\end{enumerate}
\end{lemma}

In order to establish the iteration complexity of Algorithm~\ref{alg:inexact_withlinesearch}, we first present a sufficient condition on the degree of the inexactness of the gradient and Hessian. 
\begin{condition}\label{cond:opt_epsilong_epsilonh} 
We require the inexact gradient $\g_k$ and Hessian $\H_k$ to satisfy Condition \ref{cond:appr_gh} with 
\[
\delta_{g, k} \leq \frac{1-\zeta}{8}\max\Big(\epsilong, \min\left(\epsilonh\|\d_k\|, \|\g_k\|, \|\g_{k+1}\|\right)\Big),~~~\text{and}~~~\delta_H \leq \left(\frac{1-\zeta}{4}\right)\epsilon_H.
\]
\end{condition}

One could simplify Condition \ref{cond:opt_epsilong_epsilonh}  to have an iteration-independent condition on $ \delta_{g,k} \equiv \delta_{g} $, namely,
\begin{align*}
	\delta_{g} \leq \frac{1-\zeta}{8} \epsilong.
\end{align*}
However, the adaptivity of the iteration-dependent version of Condition \ref{cond:opt_epsilong_epsilonh} through $\g_k$ and $\g_{k+1}$ offers practical advantages. 
Indeed, in many iterations, one can expect $\| \g_k \|$ and $\| \g_{k+1}\|$ to be of similar magnitudes. 
Also, as shown in Lemma~\ref{lemma:d_from_cappedcg}, we have $\|\d_k\| \leq 1.1 \epsilonh^{-1} \|\g_k\|$. 
Thus, the three terms in $\min(\epsilonh\|\d_k\|, \|\g_k\|, \|\g_{k+1}\|)$ are often  roughly of the same order, and usually larger than $\epsilong$.
These observations suggest that when the true gradient is large, we can employ loose approximations. 

Given Condition~\ref{cond:opt_epsilong_epsilonh}, the proofs of the complexity bounds boil down to three parts. 
First, we bound the decrease in the objective function $f(\x_k)$ (Lemma~\ref{lemma:dtype_sol_opt}) when taking the damped Newton step $\d_k$ (that is, when $\dtype=\SOL$ on return from Procedure~\ref{alg:capped_cg} and $\|\d_k\|$ is not too small). 
Second, we bound the decrease in the objective when a negative curvature direction is encountered in Procedure~\ref{alg:capped_cg} (Lemma~\ref{lemma:nc_from_cappedcg}) or Procedure~\ref{alg:minimum_eigenvalue} (Lemma~\ref{lemma:nc_both_procedure}). 
Third, for Lines 9-18 in Algorithm~\ref{alg:inexact_withlinesearch}, we show that the algorithm can be terminated after the update in Line \ref{line:term1}. 
In particular, when the update direction is sufficiently small from Procedure~\ref{alg:capped_cg} and a large negative curvature from Procedure~\ref{alg:minimum_eigenvalue} has not been detected, Line \ref{line:term1} terminates at a point satisfying the required optimality conditions (Lemma~\ref{lemma:terminate_step_condition_withlinesearch}). 

We start with the case in which an inexact Newton step is used.
\begin{lemma}
\label{lemma:dtype_sol_opt}
Suppose that \cref{ass:lip} is satisfied and that Condition \ref{cond:opt_epsilong_epsilonh} holds for all $k$.
Suppose that at iteration $k$ of \cref{alg:inexact_withlinesearch}, we have $\|\g_k\| \ge \epsilong$, so that Procedure~\ref{alg:capped_cg} is called. 
When Procedure~\ref{alg:capped_cg} outputs a direction $\d_k$ with $\dtype=\SOL$ and $\norm{\d_k} > {\epsilong}/\epsilonh$, 
then the backtracking line search requires at most $j_k\leq \jsol+1$ iterations, where
\[
    \jsol = \left\lceil \frac{1}{2} \log_\theta \left( \frac{3(1 - \zeta)\epsilonh^2}{4.4 U_g(L_H + \eta)} \right) \right\rceil,
\]
and the resulting step $\x_{k+1}=\x_k+\alpha_k\d_k$ satisfies 
\begin{equation} \label{eq:csol}
    f(\x_k) -f(\x_{k+1}) \geq \csol \max\left\{0, \min \left( \frac{(\|\g_{k+1} \| - \delta_{g,k} - \delta_{g,k+1})^3}{(2.5 \epsilonh)^3}, (2.5 \epsilonh)^3, \epsilong^{3/2}\right)\right\},
\end{equation}
where
\[
    \csol = \frac{\eta}{6} \min\left\{ \frac{1}{(1+2L_H)^{3/2}}, \left[ \frac{3\theta^2(1 - \zeta)}{4(L_H + \eta)}\right]^{3/2}  \right\}.
\]
\end{lemma}
\begin{proof} 
When the $\dtype=\SOL$, $\d_k$ is the solution of the inexact regularized Newton equations. 
We first prove that when $\d_k^T\g_k < 0$, the inner product $\d_k^T\nabla f(\x_k)$ is also negative:
\begin{equation*}
\begin{aligned}
    \d_k^T \nabla f(\x_k) &\leq \d_k^T\g_k + \delta_{g,k}\|\d_k\| \\
    & = \d_k^T \hat \rr_k - \d_k^T(\H_k+2\epsilonh\I)\d_k  + \delta_{g,k} \|\d_k\| && (\mbox{from \cref{eq:r_k}})\\
    & \leq \|\d_k\| \|\hat r_k\| - \epsilonh \|\d_k\|^2 + \delta_{g,k}\|\d_k\| &&(\mbox{from \cref{eqn:pos}})\\
    & \leq \frac12 \epsilonh \zeta \| \d_k \|^2  - \epsilonh \|\d_k\|^2 + \delta_{g,k}\|\d_k\| && (\mbox{from \cref{eqn:dk_sol_rk}}) \\
    &\leq -\frac{1}{2}\epsilonh \|\d_k\|^2 + \frac{1-\zeta}{8}\max\left(\epsilong, \epsilonh\|\d_k\|\right) \|\d_k\| && (\mbox{from $\zeta \in (0,1)$ and Condition~\ref{cond:opt_epsilong_epsilonh}}) \\
    & = -\frac{1}{2}\epsilonh \|\d_k\|^2 + \frac{1-\zeta}{8}\epsilonh\|\d_k\|^2 && (\mbox{from $\norm{\d_k} > {\epsilong}/\epsilonh$})\\
    &< -\frac{3}{8}\epsilonh \|\d_k\|^2.
\end{aligned}
\end{equation*}

We consider two cases here.

\textbf{Case 1:} Consider first the case in which the value $\alpha_k=1$ is accepted by the backtracking line search procedure. 
We first note that in the case $\| \g_{k+1} \| - \delta_{g,k} - \delta_{g,k+1} \le 0$, the claim \eqref{eq:csol} is satisfied trivially, because $f(\x_{k+1})<f(\x_k)$ and the right-hand side of \eqref{eq:csol} is $0$.
Thus we assume in the rest of the argument for this case that $\| \g_{k+1} \| - \delta_{g,k} - \delta_{g,k+1} > 0$. We have 
\begin{align*}
& \|\g_{k+1}\| = \|\g_{k+1}-\g_k + \g_k\| \\
& = \|\g_{k+1}-\nabla f_{k+1} + \nabla f_{k+1} - \g_k - \nabla f_k + \nabla f_{k} -\nabla^2 f(\x_k) \d_k - 2\epsilonh \d_k + \nabla^2 f(\x_k) \d_k - \H_k\d_k + \hat\rr_k \|\\
&\leq \delta_{g,k} + \delta_{g,k+1} + \|\nabla f_{k+1}-\nabla f_k -\nabla^2 f(\x_k) \d_k\|+\| 2\epsilonh \d_k\|+\|  \nabla^2 f(\x_k) \d_k - \H_k\d_k\|+\|\hat\rr_k \|\\
&\leq \delta_{g,k} + \delta_{g,k+1} + \frac{L_H}2 \|\d_k\|^2 + 2\epsilonh\|\d_k\| + \deltah\|\d_k\| + \frac12 \epsilonh\zeta\|\d_k\| \quad \quad  (\mbox{from \cref{eqn:dk_sol_rk}})\\
& =  \delta_{g,k} + \delta_{g,k+1} + \left(2\epsilonh + \deltah + \frac12 \epsilonh\zeta \right)\|\d_k\| + \frac{L_H}2 \|\d_k\|^2 \\
& \leq \delta_{g,k} + \delta_{g,k+1} + \left(2\epsilonh + \frac{1-\zeta}{2}\epsilonh  + \frac12 \epsilonh\zeta \right)\|\d_k\| + \frac{L_H}2 \|\d_k\|^2  \quad \quad (\mbox{from Condition~\ref{cond:opt_epsilong_epsilonh}}) \\
&= \delta_{g,k} + \delta_{g,k+1} + 2.5 \epsilonh \|\d_k\| + \frac{L_H}2 \|\d_k\|^2 .
\end{align*}

We thus have $A \|\d_k \|^2 + B\|\d_k \| - C \ge 0$, where $A=L_H/2$, $B=2.5 \epsilonh$, and $C = \| \g_{k+1}\| - \delta_{g,k} - \delta_{g,k+1} > 0$. Since for any $D \ge 0$ and $t \ge 0$ we have $-1 + \sqrt{1+Dt} \ge \left(-1 + \sqrt{1+D}\right) \min\left\{t,1\right\}$ (see~\citet[Lemma 17]{royer2018complexity}),  it follows that
\begin{align*}
\| \d_k \| \ge \frac{-B + \sqrt{B^2+4AC}}{2A} & = \left( \frac{-1+ \sqrt{1+4AC/B^2}}{2A} \right) B \ge \left( \frac{-1+\sqrt{1+4A}}{2A} \right) \min \left\{C/B,B\right\}\\
& = \left( \frac{2}{\sqrt{1+4A}+1} \right) \min \left\{C/B,B\right\} \ge \left( \frac{1}{\sqrt{1+4A}} \right) \min \left\{C/B,B\right\},
\end{align*}
where the last step follows from $A>0$.
By substituting for $A$, $B$, and $C$, we obtain
\[
    \|\d_k\| \geq \frac{1}{\sqrt{1+2L_H}} \min\left\{\frac{\|\g_{k+1}\|-\delta_{g,k} - \delta_{g,k+1}}{2.5 \epsilonh}, 2.5 \epsilonh\right\}.
\]
Since $\alpha_k=1$ was accepted by the backtracking line search, we have
\begin{align*}
     f(\x_k) - f(\x_k+\d_k) &\geq \frac{\eta}6 \|\d_k \|^3 \\
    & \ge \frac{\eta}6 \frac{1}{(1+2L_H)^{3/2}} \min\left\{\frac{(\|\g_{k+1}\|-\delta_{g,k} - \delta_{g,k+1})^3}{(2.5 \epsilonh)^3}, (2.5 \epsilonh)^3 \right\}.
\end{align*}
By combining this inequality with the trivial inequality obtained when $\|\g_{k+1}\|-\delta_{g,k} - \delta_{g,k+1} \le 0$, we obtain \cref{eq:csol} for the case of $\alpha_k=1$.

\textbf{Case 2:} 
As a preliminary step, note that for any $\alpha \in [0,1]$, we have the following:
\begin{align}
\nonumber
& \alpha\g_k^T\d_k + \tfrac12{\alpha^2} \d_k^T \H_k\d_k  \\
\nonumber
&= \alpha \left[ \hat\rr_k - (\H_k + 2 \epsilonh I) \d_k \right]^T \d_k 
+ \tfrac12{\alpha^2} \d_k^T \H_k\d_k   \quad \quad  (\mbox{from \cref{eq:r_k}})  \\
\nonumber
& \le \alpha \| \hat\rr_k \| \| \d_k \| 
- \alpha \left( 1- \tfrac12 \alpha \right)  \d_k^T (\H_k+2 \epsilonh I) \d_k 
- \alpha^2\epsilonh \| \d_k \|^2\\
\nonumber
& \le \alpha \| \hat\rr_k \| \| \d_k \| 
- \alpha \left( 1- \tfrac12 \alpha \right)  \d_k^T (\H_k+2 \epsilonh I) \d_k && \\
\nonumber
& \le \tfrac12 \alpha \epsilonh \zeta \|\d_k \|^2 - \tfrac12 \alpha \epsilonh \| \d_k \|^2 \quad \quad (\mbox{from $1-\tfrac12 \alpha \ge \tfrac12$, \cref{eqn:pos}, and \cref{eqn:dk_sol_rk}})\\
\label{eq:ic1}
&= \tfrac12{\alpha} \epsilonh (\zeta-1) \|\d_k \|^2.
\end{align}

\reftwo{Now consider the case where $ \alpha_k = 1 $ is not accepted by the line search. 
In this case, suppose $j \ge 0$ is the largest integer such that the step acceptance condition is not satisfied.} 
For this $j$, we have the following:
\begin{alignat*}{2}
& -\frac{\eta}6 \theta^{3j} \|\d_k\|^3 \\
&\leq f(\x_k + \theta^j\d_k) - f(\x_k) \\
&\leq \theta^{j}\nabla f_k^T\d_k + \frac{\theta^{2j}}{2} \d_k^T\nabla^2 f(\x_k) \d_k + \frac{L_H}6 \theta^{3j} \|\d_k\|^3 \;\; && (\mbox{from \cref{eq:lipH2}})\\
& \le \theta^{j}\g_k^T\d_k + \frac{\theta^{2j}}{2} \d_k^T \H_k\d_k + \theta^j\delta_{g,k}\|\d_k\| + \frac{\theta^{2j}}{2}\deltah\|\d_k\|^2 + \frac{L_H}6 \theta^{3j}\|\d_k\|^3  \;\; && (\mbox{from Definition~\ref{cond:appr_gh}})\\
& \leq -\frac{\theta^j}2 (1-\zeta)\epsilonh\|\d_k\|^2 + \theta^j\delta_{g,k}\|\d_k\| + \frac{\theta^{2j}}{2}\deltah\|\d_k\|^2 + \frac{L_H}6 \theta^{3j}\|\d_k\|^3  \; && (\mbox{from \cref{eq:ic1}})\\
& \leq -\frac{\theta^j}2 \|\d_k\|^2 \big((1-\zeta)\epsilonh-\deltah\big)  + \theta^j\delta_{g,k}\|\d_k\| + \frac{L_H}6 \theta^{3j}\|\d_k\|^3\; && (\mbox{from $ 0< \theta < 1 $}). 
\end{alignat*}
By rearranging this expression, we obtain
\[
    \theta^{2j} \geq \left( \frac{3}{L_H+\eta} \right) \left( \frac{\big((1-\zeta)\epsilonh-\deltah\big) \|\d_k\| - 2 \delta_{g,k}}{\|\d_k\|^2} \right).
\]
From  Condition~\ref{cond:opt_epsilong_epsilonh}, we have $\delta_H \le {(1-\zeta)} \epsilon_H/2$, so this bound implies that
\begin{equation} \label{eq:dj1}
    \theta^{2j} \geq \left( \frac{3}{L_H+\eta} \right) \frac{(1-\zeta)\epsilonh \|\d_k\| - 4\delta_{g,k}}{2 \|\d_k\|^2}.
\end{equation}
Since by assumption  $\norm{\d_k} \ge {\epsilong}/{\epsilonh}$,  we have from Condition~\ref{cond:opt_epsilong_epsilonh} that either
\begin{equation}\label{eq:diff2}
\delta_{g,k} \le \frac{1-\zeta}{8} \epsilong = \frac{1-\zeta}{8} \epsilonh {\frac{\epsilong}{\epsilonh}}\le \frac{(1-\zeta)\epsilonh\norm{\d_k}}{8}, 
\end{equation}
or else
\begin{equation}\label{eq:diff2_2}
\delta_{g,k} \le \frac{1-\zeta}8 \min(\epsilonh\|\d_k\|, \|\g_k\|, \|\g_{k+1}\|) < \frac{(1-\zeta)\epsilonh\norm{\d_k}}{8}.
\end{equation}
In either case, we have that $(1-\zeta) \epsilonh \|\d_k \| - 4\delta_{g,k} \ge (1-\zeta) \epsilonh \|\d_k \|/2$, so we have from \eqref{eq:dj1} that 
\begin{equation} \label{eq:dj2}
    \theta^{2j} \geq \left( \frac{3}{L_H+\eta} \right)\left( \frac{(1 - \zeta)\epsilon_H}{4\norm{\d_k}}\right).
\end{equation}
Since in the case under consideration, the acceptance condition for the backtracking line search fails for $j=0$, the latter expression holds with $j=0$, and we have
\begin{align}\label{eq:dk_lowerbound2}
    \norm{\d_k} \ge \frac{3(1-\zeta)\epsilonh}{4(L_H + \eta)}.
\end{align}
From \eqref{eq:dj2}, \eqref{eqn:dk_sol_upper}, and \eqref{eq:bounds}, we know that
\begin{equation} \label{eq:dj3}
    \theta^{2j} \ge \frac{3(1-\zeta)\epsilonh}{4(L_H + \eta)} \norm{\d_k}^{-1} \ge \frac{3(1-\zeta)\epsilonh}{4(L_H + \eta)} \frac{\epsilonh}{1.1 U_g}.
\end{equation}
Since
\[
\jsol = \left\lceil \frac{1}{2} \log_\theta \frac{3(1 - \zeta)\epsilonh^2}{4.4 U_g(L_H + \eta)} \right\rceil,
\]
then for any $j > \jsol$, we have 
\[
\theta^{2j} < \theta^{2\jsol} \le \frac{3(1 - \zeta)\epsilonh^2}{4.4 U_g(L_H + \eta)}.
\]
By comparing this expression with \eqref{eq:dj3}, we conclude that the line-search acceptance condition cannot be rejected for $j >\jsol$, so the step taken is $\alpha_k = \theta^{j_k}$ for some $j_k \le \jsol+1$. From \eqref{eq:dj3}, the preceding index $j=j_k-1$ satisfies
\[
\theta^{2j_k -2} \ge \frac{3(1 - \zeta)\epsilonh}{4(L_H + \eta)} \norm{\d_k}^{-1},
\]
so that 
\[
\theta^{j_k} \ge \sqrt{\frac{3\theta^2(1 - \zeta)}{4(L_H + \eta)}} \epsilonh^{1/2}\norm{\d_k}^{-1/2}.
\]
Then, we have
\begin{align}
\nonumber
f(\x_k) - f(\x_k + \theta^{j_k}\d_k) & \ge \frac{\eta}{6} \theta^{3j_k} \norm{\d_k}^3 \\
\nonumber
& \ge \frac{\eta}{6} \left[ \frac{3\theta^2(1 - \zeta)}{4(L_H + \eta)}\right]^{3/2} \epsilonh^{3/2} \norm{\d_k}^{3/2}  \\
& \ge \frac{\eta}{6} \left[ \frac{3\theta^2(1 - \zeta)}{4(L_H + \eta)}\right]^{3/2} \epsilong^{3/2},
\label{eq:sol_lower_bound_improvement}
\end{align}
where the last inequality follows from $\|\d_k \| \ge \epsilong/\epsilonh$.
% \cref{eq:dk_lowerbound2}. 

We obtain the result by combining the two cases above.
\end{proof}

Next, we deal with the negative curvature directions, for which  $\dtype=\NC$ \refboth{and for which a backtracking  birectional line search is used}. 
Lemmas~\ref{lemma:nc_from_cappedcg} and \ref{lemma:nc_both_procedure} bound the amount of decrease obtained from the negative curvature directions obtained  in Procedures~\ref{alg:capped_cg} and~\ref{alg:minimum_eigenvalue}, respectively. 

\begin{lemma}
\label{lemma:nc_from_cappedcg}
Suppose that Assumption~\ref{ass:lip} is satisfied and that Condition~\ref{cond:opt_epsilong_epsilonh} holds for all $k$.
Suppose that at iteration $k$ of Algorithm~\ref{alg:inexact_withlinesearch}, we have $\|\g_k\| \ge \epsilong$, so that Procedure~\ref{alg:capped_cg} is called. 
When Procedure~\ref{alg:capped_cg} outputs a direction $\d_k$ with $\dtype=\NC$ that is subsequently used as a search direction, the backtracking \refboth{birectional line search terminates with \eqref{eq:suffdecr} satisfied by either  $\alpha_k =\theta^{j_k}$ or $\alpha_k =-\theta^{j_k}$, with $j_k\leq \jnc+1$,} where
\[
    \jnc = \left\lceil \log_\theta \frac{3}{2(L_H+\eta)} \right\rceil.
\]
The resulting step $\x_{k+1}=\x_k+\alpha_k\d_k$ satisfies
\[
    f(\x_k) -f(\x_{k+1}) \geq \cnc \epsilonh^3,
\]
where
\[
    \cnc =  \frac{\eta}{6} \min \left\{ \left[\frac{3\theta}{2(L_H+\eta)}\right]^3,1\right\}.
\]
\end{lemma}
\begin{proof} 
Note first that by \cref{eqn:dk_nc}, we have $\|\d_k\| = |\d^T\H_k\d| \ge \epsilonh$.
Thus, if \refboth{$\alpha_k=\pm 1$, we have by \eqref{eq:suffdecr} that} $f(\x_k) -f(\x_{k+1}) \geq \frac{\eta}{6} \| \d_k \|^3 \ge \frac{\eta}{6}  \epsilonh^3$, so the result holds in this case.

When \refboth{$|\alpha_k|<1$}, using \cref{eqn:dk_nc} again, we have 
\[
    \d_k^T\H_k\d_k = -\|\d_k\|^3 \leq -\epsilonh\|\d_k\|^2.
\]
We have from Definition~\ref{cond:appr_gh} that 
\[
    |\d_k^T(\H_k-\nabla^2 f(\x_k))\d_k | \leq \deltah \|\d_k\|^2,
\]
so by combining the last two expressions, we have
\begin{equation} \label{eq:rd9}
  \d_k^T\nabla^2 f(\x_k)\d_k  \leq -\|\d_k\|^3 + \deltah \|\d_k\|^2.
\end{equation}
\refboth{Let $j \ge 0$ be an integer such that neither $\theta^j$ nor $-\theta^j$ satisfies the criterion \eqref{eq:suffdecr}.
Supposing first that $\nabla f(\x_k)^T\d_k\leq 0$, we have 
from \eqref{eq:lipH2}  and \eqref{eq:rd9} that}
\begin{align}
\nonumber
-\frac{\eta}6 \theta^{3j} \|\d_k\|^3
&\leq f(\x_k + \theta^j\d_k) - \f(\x_k) \\
\nonumber
&\leq \theta^{j}\nabla f(\x_k)^T\d_k + \frac{\theta^{2j}}{2} \d_k^T\nabla^2 f(\x_k) \d_k + \frac{L_H}6 \theta^{3j} \|\d_k\|^3\\
\label{eq:sj88}
&\leq -\frac{\theta^{2j}}{2}\|\d_k\|^3 + \frac{\theta^{2j}}{2}\deltah\|\d_k\|^2  + \frac{L_H}6 \theta^{3j} \|\d_k\|^3.
\end{align}
\refboth{Supposing instead that  $\nabla f(\x_k)^T\d_k> 0$, we have by considering the step $-\theta^j$ that
\begin{align*}
-\frac{\eta}6 \theta^{3j} \|\d_k\|^3
&\leq f(\x_k - \theta^j\d_k) - \f(\x_k) \\
&\leq - \theta^{j}\nabla f(\x_k)^T\d_k + \frac{\theta^{2j}}{2} \d_k^T\nabla^2 f(\x_k) \d_k + \frac{L_H}6 \theta^{3j} \|\d_k\|^3\\
&\leq -\frac{\theta^{2j}}{2}\|\d_k\|^3 + \frac{\theta^{2j}}{2}\deltah\|\d_k\|^2  + \frac{L_H}6 \theta^{3j} \|\d_k\|^3,
\end{align*}
yielding the same inequality as \eqref{eq:sj88}.
After rearrangement of this inequality} and using $\|\d_k \| \ge \epsilonh$, it follows that
\begin{equation} \label{eq:fp2}
    \theta^j \geq \left(\frac{6}{L_H+\eta}\right)\left( \frac{\|\d_k\| - \deltah}{2\|\d_k\|}\right) = \frac{3}{L_H+\eta} - \frac{3\deltah}{(L_H+\eta)\|\d_k\|} \geq \frac{3}{L_H+\eta} - \frac{3\delta_H}{(L_H+\eta)\epsilon_H}.
\end{equation}
Since from Condition~\ref{cond:opt_epsilong_epsilonh}, we have $\delta_H \le (1-\zeta)\epsilonh/4 < \epsilonh/4$, then 
\begin{equation} \label{eq:rd8}
    \theta^j \ge \frac{3}{2(L_H + \eta)}.
\end{equation}
Meanwhile, we have for $j > \jnc$ that
\[
\theta^j < \theta^{\jnc} \le \frac{3}{2(L_H + \eta)}.
\]
\refboth{The last two inequalities together imply that $j \le \jnc$,} so the line search must terminate with \refboth{$\alpha_k = \pm \theta^{j_k}$} for some $j_k \le \jnc+1$. Since \eqref{eq:rd8} must hold for $j=j_k-1$, we have
\[
    \theta^{j_k-1} \geq  \frac{3}{2(L_H+\eta)} \implies \refboth{|\alpha_k|} = \theta^{j_k} \ge \frac{3\theta}{2(L_H+\eta)}.
\]
Thus, from the step acceptance condition \eqref{eq:suffdecr} together with \cref{eqn:dk_nc} and the definition of $\cnc$, we have
\[
f(\x_k) -f(\x_{k+1}) \geq \frac{\eta}{6} \refboth{|\alpha_k|^3} \| \d_k \|^3 \ge \cnc \epsilonh^3,
\]
so the required claim also holds in the case of \refboth{$|\alpha_k|<1$}, completing the~proof.
\end{proof}

We now turn our attention to the property of Procedure~\ref{alg:minimum_eigenvalue}. 
The following lemma shows that when a negative curvature direction is obtained from Procedure~\ref{alg:minimum_eigenvalue}, we can guarantee descent in the function in a similar fashion to Lemma~\ref{lemma:nc_from_cappedcg}.

\begin{lemma}\label{lemma:nc_both_procedure}
Suppose that Assumption~\ref{ass:lip} is satisfied and that Condition~\ref{cond:opt_epsilong_epsilonh} holds for all $k$.
Suppose that at iteration $k$ of Algorithm~\ref{alg:inexact_withlinesearch}, the search direction $\d_k$ is a negative curvature direction for $\H_k$, obtained from Procedure~\ref{alg:minimum_eigenvalue}.
Then the \refboth{backtracking bidirectional} line search terminates with step size \refboth{either $\alpha_k = \theta^{j_k}$ or $\alpha_k = -\theta^{j_k}$}  with $j_k \le \jnc + 1$ where $\jnc$ is defined as in Lemma~\ref{lemma:nc_from_cappedcg}.
Moreover, the decrease in function value resulting from the chosen step size satisfies 
\begin{equation}
f(\x_k) - f(\x_k + \alpha_k \d_k) \ge \frac{\cnc}{8} \epsilon_H^3, 
\end{equation}
where $\cnc$ is defined in Lemma~\ref{lemma:nc_from_cappedcg}.
\end{lemma}
\begin{proof}
Note that 
\[
\d_k^T\H\d_k \leq - \|\d_k\|^3 \leq - \frac{\epsilonh}{2}\|\d_k\|^2,
\]
so that $\|\d_k\|\geq \epsilonh/2$.  In the first part of the proof, for the case \refboth{$\alpha_k= \pm 1$}, we have
\[
f(\x_k) - f(\x_{k+1}) \ge \frac{\eta}{6} \| \d_k \|^3 \ge \frac{\eta}{6} \frac{1}{8} \epsilon_H^3 \ge \frac{\cnc}{8} \epsilon_H^3,
\]
so the result holds in this case.
The analysis of the case \refboth{$|\alpha_k|<1$} proceeds as in the proof of Lemma~\ref{lemma:nc_from_cappedcg}  until the lower bound on $\theta^j$ in \eqref{eq:fp2}, where because of $\| \d_k \| \ge \epsilon_H/2$, we have
\[
\theta^j \ge \frac{3}{L_H+\eta} - \frac{6\delta_H}{(L_H+\eta)\epsilon_H},
\]
which, because of $\delta_H \le \epsilon_H/4$, still yields the lower bound \eqref{eq:rd8}, allowing the result of the proof to proceed as in the earlier result, except for the factor of $1/8$.
\end{proof}

Now comes a crucial step. When the output direction $\d_k$ from Procedure~\ref{alg:capped_cg} satisfies $\|\d_k\|\leq \epsilong/\epsilonh$ and Procedure~\ref{alg:minimum_eigenvalue} detects no  significant negative curvature in the Hessian,
the update of $\x_k$ with unit step along $\d_k$ is the final step of Algorithm~\ref{alg:inexact_withlinesearch}.
Dealing with this case is critical to obtaining the convergence rate of our inexact damped Newton-CG algorithm. 
\begin{lemma}
\label{lemma:terminate_step_condition_withlinesearch}
Suppose that Assumption~\ref{ass:lip} is satisfied and that Condition~\ref{cond:opt_epsilong_epsilonh} holds for all $k$.
Suppose that \cref{alg:inexact_withlinesearch} terminates at iteration $k$ at line~\ref{line:term1}, and returns $\x_k+\d_k$, where $\d_k$ is obtained from Procedure~\ref{alg:capped_cg} and satisfies $\norm{\d_k} \le {\epsilong}/{\epsilonh}$. Then we have
\[
\norm{\nabla f(\x_k + \d_k)} \le \frac{L_H}2\dfrac{\epsilong^2}{\epsilonh^2} + 4 \epsilong.
\]
If in addition the property $\H_k \succeq -\epsilonh I$ holds, then 
\[
\lambda_{\min}(\nabla^2 f(\x_k + \d_k)) \ge -\left(\frac54 \epsilon_H + L_H\dfrac{\epsilong}{\epsilonh} \right) I.
\]
\end{lemma}
\begin{proof} Note that termination at line~\ref{line:term1} occurs only if $\dtype=\SOL$, so Part 1 of Lemma~\ref{lemma:d_from_cappedcg} holds. For the gradient norm at $\x_k+\d_k$, we have 
\begin{align*}
\norm{\nabla f(\x_k + \d_k)} 
&\le  \norm{\nabla f(\x_k+\d_k) - \nabla f(\x_k)  - \nabla^2 f(\x_k) \d_k + \H_k \d_k+ \g_k}
\\ & \quad\quad +\norm{\nabla f(\x_k) - \g_k}  + \norm{\nabla^2 f(\x_k)\d_k - \H_k\d_k}
\\
& \le \norm{\nabla f(\x_k+\d_k) - \nabla f(\x_k)  - \nabla^2 f(\x_k) \d_k } + \norm{\H_k\d_k + \g_k} + \delta_{g,k} + \delta_H\norm{\d_k} 
\\
& \le \norm{\nabla f(\x_k+\d_k) - \nabla f(\x_k)  - \nabla^2 f(\x_k) \d_k } + \norm{\hat\rr_k} + 2\epsilonh \norm{\d_k} + \delta_{g,k} + \delta_H\norm{\d_k} 
\\
& \le \frac{L_H}2\norm{\d_k}^2 + \frac{1}{2}\epsilonh \zeta\norm{\d_k} + (2\epsilonh + \delta_H)\norm{\d_k} + \delta_{g,k} \quad  \quad  \mbox{(from \eqref{eq:lipH} and \eqref{eqn:dk_sol_rk})}
\\
& \le \frac{L_H}2\norm{\d_k}^2 + 3\epsilonh \norm{\d_k} + \delta_{g,k} 
\quad \quad  \quad  \quad    \mbox{(since $\zeta \in (0,1)$ and $\deltah \le \epsilonh/2$)}
\\ 
& \le \frac{L_H}2\norm{\d_k}^2 + 3\epsilonh \norm{\d_k} + \left( \frac{1-\zeta}{8} \right) \max\left(\epsilong, \min(\epsilonh\|\d_k\|, \|\g_k\|, \|\g_{k+1}\|)\right) 
\\
& \le \frac{L_H}2\dfrac{\epsilong^2}{\epsilonh^2} + 3 \epsilonh \norm{\d_k}  + \frac{1-\zeta}{8}\max\left(\epsilong, \epsilonh\|\d_k\|\right) 
\\
& \le \frac{L_H}2\dfrac{\epsilong^2}{\epsilonh^2} + 3\epsilong + \frac{1-\zeta}{8} \epsilong \quad \quad \quad  \quad   
\mbox{(since $\| \d_k \| \le \epsilong/\epsilonh$)}
\\
& \le \frac{L_H}2\dfrac{\epsilong^2}{\epsilonh^2} + 4 \epsilong,
\end{align*}
as required.

For the second-order condition, since $\H_k \succeq -\epsilon_H I$ and $\deltah \le \epsilonh/4$ (from Condition~\ref{cond:opt_epsilong_epsilonh}), we have
\[
\nabla^2 f(\x_k + \d_k) \succeq \nabla^2 f(\x_k) - L_H\norm{\d_k} I
 \succeq \H_k - \delta_H I -  L_H\dfrac{\epsilong}{\epsilonh} I  \succeq -\left(\frac54 \epsilon_H + L_H\dfrac{\epsilong}{\epsilonh} \right) I.
\]
This completes the proof.
\end{proof}

Now, combining Lemmas \ref{lemma:dtype_sol_opt}--\ref{lemma:terminate_step_condition_withlinesearch}, we obtain the iteration complexity for \cref{alg:inexact_withlinesearch}.

\begin{theorem}
\label{thm:opt_iteration_comp}
Suppose that Assumption~\ref{ass:lip} is satisfied and that Condition~\ref{cond:opt_epsilong_epsilonh} holds for all $k$.
For a given $\epsilon>0$, let $\epsilonh = \sqrt{L_H \epsilon}, \epsilong = \epsilon$. 
Define 
\begin{equation}\label{eq:k}
\bar K := 
\left\lceil \frac{3(f(\x_0)- f_\text{low})}{\min  \left(  \frac{1}{64 L_H^{3/2}} \csol, 8 L_H^{3/2} \csol,  L_H^{3/2} \cnc/8  \right)} \epsilon^{-3/2} \right \rceil + 5,
\end{equation}
where $\csol$ and $\cnc$ are defined in Lemmas \ref{lemma:dtype_sol_opt} and \ref{lemma:nc_from_cappedcg}, respectively. 
Then \cref{alg:inexact_withlinesearch} terminates in at most $\bar{K}$ iterations at a point satisfying 
\[
\norm{\nabla f(\x)} \lesssim \epsilon.
\]
Moreover, with probability at least $(1 - \delta)^{\bar K }$ the point returned by \cref{alg:inexact_withlinesearch} also satisfies the approximate second-order condition
\begin{equation} \label{eqn:2oeps}
\lambda_{\min}(\nabla^2 f(\x)) \gtrsim -\sqrt{L_H\epsilon}.
\end{equation}
\reftwo{Here, $\lesssim$ and $\gtrsim$ denote that the corresponding inequality holds up to a certain constant that is independent of $ \epsilon $ and $ L_{H} $.}
\end{theorem}
\begin{proof}
Note first that for our choices of $\epsilong$ and $\epsilonh$, the threshold $\epsilong/\epsilonh$ for $\| \d_k \|$ in line~\ref{line:threshold} of Algorithm~\ref{alg:inexact_withlinesearch} becomes $\sqrt{\epsilon/L_H}$.

We show first that \cref{alg:inexact_withlinesearch} terminates after at most $\bar K$ steps. We taxonomize the iterations into five classes. To specify these classes, we denote by $\d_k$ and $\dtype$ the values of these variables {\em immediately before a step is taken or termination is declared}, bearing in mind that these variables can be reassigned during iteration $k$, in Line \ref{line:newdk}.
Supposing for contradiction that \cref{alg:inexact_withlinesearch} runs for at least $K$ steps, for some $K>\bar{K}$, we define the five classes of indices as follows.
\begin{align*}
    \mathcal K_1 &:= \{k=0,1,2,\dotsc,  K-1 \, | \, \norm{\g_k} < \epsilon \} \\
    \mathcal K_2 &:= \{k=0,1,2,\dotsc,  K-1 \, | \, \norm{\g_k} \ge \epsilon, \, \dtype=\SOL, \, \|\d_k \| > \sqrt{\epsilon/L_H}, \, \norm{\g_{k+1}} < \epsilon \} \\
    \mathcal K_3 &:= \{k=0,1,2,\dotsc,  K-1 \, | \, \norm{\g_k} \ge \epsilon, \, \dtype=\SOL, \, \|\d_k \| > \sqrt{\epsilon/L_H}, \, \norm{\g_{k+1}} \ge \epsilon \} \\
    \mathcal K_4 &:= \{k=0,1,2,\dotsc,  K=1 \, | \, \norm{\g_k} \ge \epsilon, \, \dtype=\SOL, \, \| \d_k \| \le \sqrt{\epsilon/L_H}\} \\
\mathcal K_5 &:= \{k=0,1,2,\dotsc, K-1 \, | \, \norm{\g_k} \ge \epsilon, \, \dtype=\NC\}.
\end{align*}
Obviously, $K= \Abs{\mathcal K_1} + \Abs{\mathcal K_2} + \Abs{\mathcal K_3} + \Abs{\mathcal K_4} + \Abs{\mathcal K_5} $. We consider each of these types of steps in turn.

\paragraph{Case 1:} $k \in \mathcal K_1$.
The update $\d_k$ in this case must come from Procedure~\ref{alg:minimum_eigenvalue}. Either the method terminates (which happens at most once!) or from Lemma~\ref{lemma:nc_both_procedure}, we have that
\begin{equation}
    f(\x_k) - f(\x_{k+1}) \geq \frac18 \cnc \epsilonh^3 =\frac18 L_H^{3/2} \cnc \epsilon^{3/2}.
\end{equation}
Thus the total amount of decrease that results from steps in $\mathcal K_1$ is at least $(\Abs{\mathcal K_1} -1) L_H^{3/2} \cnc \epsilon^{3/2}/8$.

\paragraph{Case 2:} $k \in \mathcal K_2$.
With \cref{lemma:dtype_sol_opt}, we can guarantee only that $f(\x_k) - f(\x_{k+1}) \ge 0$. However, since $\norm{\g_{k+1}} < \epsilon$, the {\em next} iterate must belong to class $\mathcal K_1$. Therefore we have $\Abs{\mathcal K_2} \le \Abs{\mathcal K_1}$.

\paragraph{Case 3:} $k \in \mathcal K_3$.
Here the step $\d_k$ is an approximate solution of the damped Newton equations, and we can apply \cref{lemma:dtype_sol_opt} to obtain a nontrivial lower bound on the decrease in $f$.
By Condition \ref{cond:opt_epsilong_epsilonh}, we have
\begin{align*}
    \delta_{g,k} & \le \frac18 \max\left(\epsilong, \min(\epsilonh\|\d_k\|, \|\g_k\|, \|\g_{k+1}\|)\right) \le \frac18 \max( \epsilong, \| \g_{k+1} \|) = \frac18 \| \g_{k+1} \|, \\
    \delta_{g,k+1} & \le \frac18 \max\left(\epsilong, \min(\epsilonh\|\d_{k+1}\|, \|\g_{k+1}\|, \|\g_{k+2}\|)\right) \le \frac18 \max( \epsilong, \| \g_{k+1} \|) = \frac18 \| \g_{k+1} \|,
\end{align*}
so that 
\[
\|\g_{k+1} \| - \delta_{g,k} - \delta_{g,k+1} \ge \frac34 \| \g_{k+1} \| \ge \frac34 \epsilong = \frac34 \epsilon.
\]
Thus, from \eqref{eq:csol} in Lemma~\ref{lemma:dtype_sol_opt}, we have for this type of step that 
\begin{align*}
    f(\x_k) -f(\x_{k+1}) & \ge \csol \max\left\{0, \min \left( \frac{(\|\g_{k+1} \| - \delta_{g,k} - \delta_{g,k+1})^3}{(2.5 \epsilonh)^3}, (2.5 \epsilonh)^3, \epsilong^{3/2}\right)\right\} \\
    & \ge \csol \min \left( \frac{(\tfrac34 \epsilon)^3}{(2.5 \sqrt{L_H \epsilon})^3}, (2.5 \sqrt{L_H \epsilon})^3, \epsilon^{3/2}\right) \\
    & = \csol \min \left( \frac{1}{64L_H^{3/2}}, 8 L_H^{3/2}, 1  \right) \epsilon^{3/2} =
    \csol \min \left( \frac{1}{64L_H^{3/2}}, 8 L_H^{3/2} \right) \epsilon^{3/2}.
\end{align*}

\paragraph{Case 4:} $k \in \mathcal K_4$.
In this case, Procedure~\ref{alg:capped_cg} outputs $\dtype=\SOL$ along with a ``small'' value of $\d_k$. Subsequently, Procedure~\ref{alg:minimum_eigenvalue} was called, but it must have returned with a certification of near-positive-definiteness of $\H_k$, since $\dtype$ was not switched to $\NC$. 
Thus, according to \cref{lemma:terminate_step_condition_withlinesearch}, termination occurs with output $\x_k+\d_k$. Thus, this case can occur at most once, and we have $|\mathcal K_4| \le 1$.

\paragraph{Case 5:} $k \in \mathcal K_5$.
In this case, either the algorithm terminates and outputs $\x=\x_k$ (which happens at most once), or else a step is taken along a negative curvature direction for $\H_k$, detected either in Procedure \ref{alg:capped_cg} or Procedure \ref{alg:minimum_eigenvalue}. In the former case (detection in Procedure \ref{alg:capped_cg}), we have from 
\cref{lemma:nc_from_cappedcg} that
$f(\x_k) -f(\x_{k+1}) \ge \cnc \epsilonh^3 = \cnc L_H^{3/2} \epsilon^{3/2}$, while in
 the latter case (detection in Procedure \ref{alg:minimum_eigenvalue}), we have from \cref{lemma:nc_both_procedure} that
$f(\x_k) - f(\x_{k+1}) \geq  \frac18 L_H^{3/2} \cnc \epsilon^{3/2}$.
Thus, the total decrease in $f$ resulting from steps of this class is bounded below by
$(| \mathcal K_5|-1) \frac18 L_H^{3/2} \cnc \epsilon^{3/2}$.

The total decrease of $f$ over all $K$ steps cannot exceed $f(\x_0) - f_\text{low}$. We thus have
% \footnote{If $\Abs{\mathcal K_4} = 1$, we should consider the function decrease over the $K-1$ steps. The analysis on bounding $\Abs{\mathcal K_1},\Abs{\mathcal K_3}, \Abs{\mathcal K_5}$ remains the same.}
\begin{align*}
    f(\x_0) - f_\text{low} & \ge \sum_{k=0}^{K-1} (f(\x_k) - f(\x_{k+1})) \\
    & \ge \sum_{k\in \mathcal K_1}(f(\x_k) - f(\x_{k+1}))  + \sum_{k\in \mathcal K_3}(f(\x_k) - f(\x_{k+1})) + \sum_{k\in \mathcal K_5}(f(\x_k) - f(\x_{k+1})) \\
    & \ge (\Abs{\mathcal K_1} + \Abs{\mathcal K_5} -2 ) \frac18 L_H^{3/2} \cnc \epsilon^{3/2} + \Abs{\mathcal K_3} \csol \min  \left(  \frac{1}{64 L_H^{3/2}}, 8 L_H^{3/2}\right) \epsilon^{3/2}. 
\end{align*}
Therefore, we have
\begin{align*}
    \Abs{\mathcal K_1} + \Abs{\mathcal K_5} -2 & \le \frac{f(\x_0) - f_\text{low}}{L_H^{3/2} \cnc /8}\epsilon^{-3/2}, \\
    \Abs{\mathcal K_3} & \le \frac{f(\x_0)- f_\text{low}}{\csol \min  \left(  \frac{1}{64 L_H^{3/2}}, 8 L_H^{3/2}\right)} \epsilon^{-3/2}.
\end{align*}
Finally, we have
\begin{align*}
    K &= \Abs{\mathcal K_1} + \Abs{\mathcal K_2} + \Abs{\mathcal K_3} + \Abs{\mathcal K_4} + \Abs{\mathcal K_5} \\
    & \le 2\Abs{\mathcal K_1} + \Abs{\mathcal K_3} + 1 + \Abs{\mathcal K_5} \\
    & \le 2 (\Abs{\mathcal K_1} + \Abs{\mathcal K_5} -2) + \Abs{\mathcal K_3} + 5 \\
    & \le \frac{2(f(\x_0) - f_\text{low})}{L_H^{3/2} \cnc /8}\epsilon^{-3/2} 
    + \frac{f(\x_0)- f_\text{low}}{\csol \min  \left(  \frac{1}{64 L_H^{3/2}}, 8 L_H^{3/2} \right)} \epsilon^{-3/2} + 5 \\
    & \le \frac{3(f(\x_0)- f_\text{low})}{\min  \left(  \frac{1}{64 L_H^{3/2}} \csol, 8 L_H^{3/2} \csol,  L_H^{3/2} \cnc/8  \right)} \epsilon^{-3/2} + 5 \le \bar{K},
\end{align*}
which contradicts our assertion that $K> \bar{K}$. Thus \cref{alg:inexact_withlinesearch} terminates in at most $\bar{K}$ steps.

Note that if termination occurs at Line \ref{line:term2} of \cref{alg:inexact_withlinesearch}, the returned value of $\x = \x_{k}$ certainly has $\| \nabla f(\x) \| \lesssim \epsilong = \epsilon$. This is because when $\|\g_k \| \le \epsilong$, we have from Condition~\ref{cond:opt_epsilong_epsilonh} that $\delta_{g,k} \le {(1-\zeta)} \epsilong/8$, so that $\| \nabla f(\x) \| \le \|\g_k\| + \delta_{g,k} \lesssim \epsilong$.
Alternatively, if termination occurs at Line~\ref{line:term1}, for the returned value of $\x = \x_{k} + \d_{k}$ we have
\[
\norm{\nabla f(\x)} \le \frac{L_H}2\dfrac{\epsilong^2}{\epsilonh^2} + 4 \epsilong = \frac{L_H}{2} \frac{\epsilon^2}{L_H \epsilon} + 4 \epsilon = \frac92 \epsilon.
\]
Thus, the claim $\norm{\nabla f(\x)} \lesssim \epsilon$ at the termination point $\x$ holds.

We now verify the claims about probability of failure and the second-order conditions. Note that for both types of termination (at Lines~\ref{line:term1} and \ref{line:term2} of \cref{alg:inexact_withlinesearch}), Procedure~\ref{alg:minimum_eigenvalue} issues a certificate that $\lambda_{\min}(\H_k) \ge -\epsilonh$. 
Subject to this certificate being correct, we show now that our claim \eqref{eqn:2oeps} holds. When termination occurs at line~\ref{line:term1}, we have in this case from Lemma~\ref{lemma:terminate_step_condition_withlinesearch} that at the returned point $\x=\x_k+\d_k$, we have
\[
\lambda_{\min} (\nabla^2 f(\x)) \ge -\left( \frac54 \epsilonh + L_H \frac{\epsilong}{\epsilonh} \right) = -\frac94 \sqrt{L_H \epsilon},
\]
as required. For termination at Line~\ref{line:term2}, we have directly that $\lambda_{\min} (\nabla^2 f(\x)) \ge -\epsilonh  = -\sqrt{L_H \epsilon}$, again verifying the claim.

We now calculate a bound on the probability of incorrect termination, which can occur at either Line~\ref{line:term1} or Line \ref{line:term2} when Procedure~\ref{alg:minimum_eigenvalue} issues a certificate that $\lambda_{\min}(\H_k) \ge -\epsilonh$, whereas in fact $\lambda_{\min}(\H_k) < -\epsilonh$.
\reftwo{The proof is a simple adaptation from \citet[Theorem~2]{XieW21a} and \citet[Theorem~4.6]{Cur19a}, the adaptations for inexactness being fairly straightforward.  We include the argument  here for the sake of completeness.}
The possibility of such an event happening on any individual call to Procedure~\ref{alg:minimum_eigenvalue} is bounded above by $\delta$. For all iterates $k$, we denote by $\tilde{P}_k$ the probability that \cref{alg:inexact_withlinesearch} reaches iteration $k$ but $\lambda_{\min}(\H_k) < -\epsilonh$, and denote by $P_k$ the probability that \cref{alg:inexact_withlinesearch} reaches iteration $k$ but $\lambda_{\min}(\H_k) < -\epsilonh$, yet the algorithm terminates due to Procedure~\ref{alg:minimum_eigenvalue} issuing an incorrect certificate. Clearly, we have $P_k \le \delta \tilde{P}_k$ for all $k=0,1,\dotsc,\bar{K}$. Since it is trivially true for all $k$ that 
\[
\tilde{P}_k + \sum_{i=0}^{k-1} P_i \le 1,
\]
we have for all $k$ that 
\begin{equation} \label{eq:ys9}
P_k \le \delta \tilde{P}_k  \le \delta \left(  1- \sum_{i=0}^{k-1} P_i \right).
\end{equation}
Now let $M_k$ be the total number of calls to Procedure~\ref{alg:minimum_eigenvalue} that have occurred up to and including iteration $k$ of \cref{alg:inexact_withlinesearch}. We prove by induction that $\sum_{i=0}^k P_i \le 1-(1-\delta)^{M_k}$ for all $k$. For $k=0$, the claim holds trivially, both in the case of $M_0=0$ (in which case $P_0=0$) and $M_0=1$ (in which case $P_0 \le \delta$). Supposing now that the claim is true for some $k \ge 0$, we show that it continues to hold for $k+1$. If \cref{alg:inexact_withlinesearch} reaches iteration $k+1$ with $\lambda_{\min}(\H_{k+1}) < -\epsilonh$, and Procedure~\ref{alg:minimum_eigenvalue} is {\em not} called at this iteration, then $M_{k+1}=M_k$ and $P_{k+1}=0$, so by the induction hypothesis we have
\[
\sum_{i=0}^{k+1} P_i  = \sum_{i=0}^k P_i  \le 1-(1-\delta)^{M_k} = 1-(1-\delta)^{M_{k+1}},
\]
as required. In the other case in which \cref{alg:inexact_withlinesearch} reaches iteration $k+1$ with $\lambda_{\min}(\H_{k+1}) < -\epsilonh$, and Procedure~\ref{alg:minimum_eigenvalue} {\em is} called at this iteration, then $M_{k+1}=M_k+1$, so by using \eqref{eq:ys9} and the inductive hypothesis, we have
\begin{align*}
   \sum_{i=0}^{k+1} P_i  & =  \sum_{i=0}^k P_i + P_{k+1} \\
   & \le  \sum_{i=0}^k P_i + \delta \left( 1- \sum_{i=0}^k P_i \right) \\
   & = \delta + (1-\delta) \sum_{i=0}^k P_i \\
   & \le \delta + (1-\delta) \left( 1- (1-\delta)^{M_k} \right) \\
   & = 1- (1-\delta)^{M_k+1} = 1- (1-\delta)^{M_{k+1}},
\end{align*}
as required. Since $M_k \le k \le \bar{K}$ for all $k=1,2,\dotsc,\bar{K}$, we have that the probability that \cref{alg:inexact_withlinesearch} terminates incorrectly on {\em any} iteration is bounded above by $1-(1-\delta)^{\bar{K}}$. So when termination occurs, the condition \eqref{eqn:2oeps} holds at the termination point with probability at least $(1-\delta)^{\bar{K}}$, as claimed.
\end{proof}

By incorporating the complexity of Procedures~\ref{alg:capped_cg} and \ref{alg:minimum_eigenvalue}, as described in Lemmas~\ref{lemma:capped_cg_iter_bounded} and \ref{lemma:fail_min_eigen_oracle}, we can obtain an upper bound on the number of approximate gradient and approximate Hessian-vector product evaluations required during a run of \cref{alg:inexact_withlinesearch}. The iteration count for the algorithm is bounded by $\OM(\epsilon^{-3/2})$ in \cref{thm:opt_iteration_comp} and each iteration requires one approximate gradient evaluation. Additionally, each iteration of \cref{alg:inexact_withlinesearch} may require a call to Procedure~\ref{alg:capped_cg}, which by \cref{lemma:capped_cg_iter_bounded} requires $\tilde\OM(\epsilonh^{-1/2}) = \tilde\OM(\epsilon^{-1/4})$ approximate Hessian-vector products. A call to Procedure~\ref{alg:minimum_eigenvalue} may also be required on some iterations. Here, by \cref{lemma:fail_min_eigen_oracle}, $\OM(\epsilonh^{-1/2}) = \OM(\epsilon^{-1/4})$ approximate Hessian-vector products may be required also. We summarize these observations in the following corollary.
\begin{corollary}\label{cor:total_complexity}
Suppose that the assumptions of \cref{thm:opt_iteration_comp} hold. Let $\epsilong=\epsilon, \epsilonh=\sqrt{L_H\epsilon}$, and $\bar K$ be defined as in \cref{eq:k}. Then for $d$ sufficiently large relative to $\epsilon^{-1/2}$, \cref{alg:inexact_withlinesearch} terminates after at most $\tilde\OM(\epsilon^{-7/4})$  matrix-vector products with the approximate Hessians and at most $\OM(\epsilon^{-3/2})$ evaluations of approximate gradients.
With probability at least $(1-\delta)^{\bar{K}}$, it returns a point that satisfies the approximate first- and second-order conditions described in \cref{thm:opt_iteration_comp}.
\end{corollary}

\subsection{Inexact Newton-CG algorithm without line search}
\label{sec:fixed_step_size}

Although \cref{alg:inexact_withlinesearch} employs approximate gradients and Hessian \refone{at various steps, the use of backtracking line search to compute the stepsize $\alpha_k$ requires exact evaluations of the function $f$ and its gradient}. This setting has indeed been considered in some previous work, e.g., \cite{yao2018inexact,roosta2019sub}. 
When \refboth{gradient evaluation has similar computational cost to the corresponding function evaluation, we may not save much in computation by requiring only an approximate gradient.}
We show in this section that a pre-defined (``fixed'') value of the step length $\alpha_k$ can be carefully chosen to obviate the need for function evaluations.
\refone{The advantage of not requiring exact evaluations of functions is considerable, but there are disadvantages too.}
First, \refone{the computed fixed step-size is conservative, so} the guaranteed descent in the objective generally will be smaller than in \cref{alg:inexact_withlinesearch}; see \cref{lemma:dtype_sol_opt,lemma:nc_from_cappedcg,lemma:nc_both_procedure}. 
Second, our approach makes use of an approximate upper bound $L_H$ on the Lipschitz constant of the Hessian, which might not be readily available. 
Fortunately, there are many important instances (especially in machine learning) where an estimate of
$L_H$ can be obtained easily; for example, empirical risk minimization problems involving the squared loss \citep{xuNonconvexEmpirical2017} and Welsch's exponential variant \citep{zhang2019robustness}. See \cref{tab:L_H_bound} for details.

\begin{table}[!htb]
	\centering
	\caption{The upper bound of $L_H$ for some non-convex finite-sum minimization problems of the form \eqref{eq:finte_sum_problem}. Here, we consider $\{(\a_i, b_i)\}_{i=1}^n$ as training data where $\a_i\in\bbR^d$ and $b_i\in\bbR$. 
		% When we calculate the upper bound of $L_H$ for single data point, and for simplicity,  we omit the sub-script of $\a_i$ and $b_i$. 
		For Welsch's exponential function $\phi$, $\alpha$ is a positive parameter.
		% In the notation column, we use $z$ to represent the output of the inner product between $\a_i$ and $\x$, i.e., $z=\langle\a_i, \x\rangle$.
		\label{tab:L_H_bound}}
	\centering
	\begin{adjustbox}{width=1\linewidth}
		\begin{tabular}{lcccc}
			\toprule
			Problem Formulation & Predictor Function & \multicolumn{1}{m{6cm}}{\centering Upper bound of $L_H$ for single data point $(\a,b)$} & \multicolumn{1}{m{5cm}}{\centering Upper bound of $L_H$ for entire problem} \\
			\midrule
			$\sum\limits_{i=1}^n(b_i-\phi(\langle \a_i, \x\rangle))^2$ & $\phi(z) =1/{(1+e^{-z})}$
			& \multicolumn{1}{m{6cm}}{\centering$2\|\a\|^3(|b\phi'''(z)| + 3|\phi'(z)\phi''(z)|+|\phi(z)\phi'''(z)|)\leq 2(|b|+4)\|\a\|^3$} & $\displaystyle \max_{i=1,\ldots,n} \; 2(|b_{i}|+4)\|\a_i\|^3$ \\ \\
			$\sum\limits_{i=1}^n(b_i-\phi(\langle \a_i, \x\rangle))^2$ & $\phi(z) ={(e^{z}-e^{-z})}/{(e^{z}+e^{-z})}$
			& \multicolumn{1}{m{6cm}}{\centering$2\|\a\|^3(|b\phi'''(z)| + 3|\phi'(z)\phi''(z)|+|\phi(z)\phi'''(z)|)\leq 2(|b|+4)\|\a\|^3$} & $\displaystyle \max_{i=1,\ldots,n} \;  2(|b_{i}|+4)\|\a_i\|^3$ \\\\ 
			$\sum\limits_{i=1}^n\phi(b_i - \langle \a_i, \x\rangle)$ & $\phi(z) ={(1 -e^{-\alpha z^2})}/{\alpha}$ & $\|\a\|^3|\phi'''(z)|$ & $9\alpha^{3/2} \displaystyle \max_{i=1,\ldots,n} \;  \|\a_i\|^3$\\ 
			\bottomrule
		\end{tabular}
	\end{adjustbox}
\end{table}

% Although the gradient and Hessian in Algorithm~\ref{alg:inexact_withlinesearch} are already approximated,  the backtracking line-search step is still costly when $n \gg 1$. 
% Here, we present Algorithm~\ref{alg:fixed_stepsize}, as a modification to~\cref{alg:inexact_withlinesearch}, which uses a fixed (pre-defined) step size rather than line search, and as a consequence, function evaluations are no longer needed. 
% Clearly, eliminating exact functional evaluations would increase the efficiency of the overall algorithm. 

We state our variant of the Inexact Newton-CG Algorithm that does not require line search as \cref{alg:fixed_stepsize}.
Lines \ref{line:dif_01}, \ref{line:dif_02}, \ref{line:dif_03}, and \ref{line:dif_04}-\ref{line:dif_05} constitute the main differences between \cref{alg:inexact_withlinesearch,alg:fixed_stepsize}. 
\begin{algorithm}[!ht]
\caption{Inexact Newton-CG without Line Search}
\footnotesize
\label{alg:fixed_stepsize}
\begin{algorithmic}[1]
\STATE {\bf Inputs:} $\epsilon_g, \epsilon_H >0$; Parameter $\theta \in(0,1)$; Starting point $\x_0$; upper bound  on Hessian norm $U_H>0$; accuracy parameter $\zeta \in (0, \min\{1,U_H\})$;
\FOR {$k=0,1,2,\cdots$}
\IF {$\norm{\g_k}  \ge \epsilon_g$} 
\STATE Call Procedure~\ref{alg:capped_cg} with $\H = \H_k, M=U_H, \epsilon = \epsilon_H, \g = \g_k$ and accuracy parameter $\zeta$ to obtain $\d$ and $\dtype$;
\IF {$\dtype == NC$}
\STATE $\d_k \gets -\sgn(\d^T\g_k)\frac{|\d^T\H_k\d|}{\norm{\d}^2} \frac{\d}{\norm{\d}}$; \label{line:dif_01}
\ELSE
\STATE $\d_k \gets \d$;
\IF {$\norm{\d_k} \le \epsilong/\epsilonh$} 
\STATE Call Procedure~\ref{alg:minimum_eigenvalue} with $\H = \H_k, M= U_H, \epsilon= \epsilon_H$ to obtain $\v$ with $\|\v\|=1$ and $\v^T \H_k \v \le -\epsilon_H/2$ or a certificate that $\lambda_{\min}(\H_k) \ge -\epsilonh$;
\IF {Procedure \ref{alg:minimum_eigenvalue} certifies that $\lambda_{\min}(\H_k) \ge -\epsilonh$}
\STATE Terminate and return $\x_k + \d_k$;
\ELSE
\STATE $\d_k \gets -\sgn(\v^T\g_k) {|\v^T\H_k\v|} \v$ and $\dtype \gets \NC$; \label{line:dif_02}
\ENDIF
\ENDIF
\ENDIF
\ELSE
\STATE $\dtype \gets \NC$;
\STATE Call Procedure \ref{alg:minimum_eigenvalue} with $\H = \H_k, M= U_H, \epsilon= \epsilon_H$ to obtain $\v$ with $\|\v\|=1$ and $\v^T \H_k \v \le -\epsilon_H/2$ or a certificate that $\lambda_{\min}(\H_k) \ge -\epsilonh$;
\IF {Procedure \ref{alg:minimum_eigenvalue} certifies that $\lambda_{\min}(\H_k) \ge -\epsilon$}
\STATE Terminate and return $\x_k$;
\ELSE
\STATE $\d_k \gets -\sgn(\v^T\g_k) {|\v^T\H_k\v|} \v$; \label{line:dif_03}
\ENDIF
\ENDIF
\IF {$\dtype==\NC$} \label{line:dif_04}
\STATE Define $\alpha_k$ as in \cref{lemma:fixed_stepsize_nc_from_cappedcg}, to satisfy $\alpha_k \ge \frac{3}{4} \frac{\tilde\theta}{L_H+\eta}$ for some $\tilde\theta \in ((2-\sqrt{3})^2,1)$ 
\ELSE
\STATE $\alpha_k = \left[ \frac{3(1 - \zeta)}{4(L_H + \eta)}\right]^{1/2} \frac{\epsilonh^{1/2}}{\|\d_k\|^{1/2}}$ (defined in \cref{lemma:fixed_stepsize_with_sol}) 
\ENDIF \label{line:dif_05}
\STATE $\x_{k+1} \gets \x_k + \alpha_k \d_k$;

\ENDFOR
\end{algorithmic}
\end{algorithm}

The analysis of this section makes use of the following condition.

\begin{condition}\label{cond:opt_epsilong_epsilonh_fixedstepsize} 
% To make the output of Algorithm~\ref{alg:inexact_withlinesearch} satisfy Definition~\ref{def:optimality}, we require 
The inexact gradient $\g_k$ and Hessian $\H_k$ satisfy Condition \ref{cond:appr_gh} with 
% \steve{Simplified the bound on $\delta_{g,k}$. The more complicated one was the result of some incorrect analysis.}
% \zhewei{Really? Thanks a lot for checking!}
\begin{align*}
\delta_{g, k} &\leq \frac{1-\zeta}{8}\min\left(   \frac{3\epsilonh^2}{65(L_H+\eta)}, \max\big(\epsilong, \min(\epsilonh\|\d_k\|, \|\g_k\|, \|\g_{k+1}\|)\big)\right)\\
\text{and}~~~\delta_H &\leq \frac{1-\zeta}{4}\epsilon_H.
\end{align*}
Throughout this section, we fix $\epsilonh = \sqrt{L_H\epsilong}$, so that $\epsilong/\epsilonh = \sqrt{\epsilong/L_H}$.
\end{condition}

In the next three lemmas, we show that the choices of $\alpha_k$ in \cref{alg:fixed_stepsize} lead to the step length acceptance condition used in \cref{alg:inexact_withlinesearch} being satisfied, that is,
\begin{equation} \label{eq:bf1}
    -\frac{\eta}6 \alpha_k^3 \|\d_k\|^3 \geq f(\x_k + \alpha_k\d_k) - f(\x_k).
\end{equation}

We now show that the fixed step size can result in a sufficient descent in the function $f(\x_k)$ when $\dtype=\SOL$ and $\|\d_k\|\geq \sqrt{{\epsilong}/{L_H}}$. 
The following lemma can be viewed as a modification of Lemma \ref{lemma:dtype_sol_opt} with fixed step size. 

\begin{lemma}
\label{lemma:fixed_stepsize_with_sol}
Suppose that Assumption~\ref{ass:lip} is satisfied and that Condition~\ref{cond:opt_epsilong_epsilonh_fixedstepsize} holds for all $k$.
Suppose that at iteration $k$ of Algorithm~\ref{alg:fixed_stepsize}, we have $\|\g_k\| \ge \epsilong$, so that Procedure~\ref{alg:capped_cg} is called. 
When Procedure~\ref{alg:capped_cg} outputs a direction $\d_k$ with $\dtype=\SOL$ and $\|\d_k\|\geq\epsilong/\epsilonh$,  Algorithm~\ref{alg:fixed_stepsize} sets
\[
\alpha_k = \left[ \frac{3(1 - \zeta)}{4(L_H + \eta)}\right]^{1/2} \frac{\epsilonh^{1/2}}{\|\d_k\|^{1/2}}.
\]
The resulting step $\x_{k+1}=\x_k+\alpha_k\d_k$ satisfies
\[
    f(\x_k) -f(\x_{k+1}) \geq \csolb \epsilonh^3,
\]
where
\[
    \csolb = \frac{\eta}{6} \left[ \frac{3(1 - \zeta)}{4L_H(L_H + \eta)}\right]^{3/2}.
\]
\end{lemma}
\begin{proof} 
First, we prove that $\alpha_k \leq 1$. We have, using $\epsilonh = \sqrt{L_H \epsilong}$, that
\[
\alpha_k^2 
        =  \frac{3(1 - \zeta)\epsilonh}{4(L_H + \eta)\|\d_k\|} 
         \leq  \frac{3(1 - \zeta)\epsilonh^{2}}{4(L_H + \eta)\epsilong}
         = \frac{3(1 - \zeta)L_H}{4(L_H + \eta)} < 1.
\]
If we can show that \eqref{eq:bf1} holds, then we obtain the conclusion of the lemma by substituting the formula for $\alpha_k$ into this expression and using $\| \d_k \| \ge \epsilong/\epsilonh$ and $\epsilonh = \sqrt{L_H\epsilong}$.

Suppose for contradiction that condition \eqref{eq:bf1} is not satisfied. Then we have 
\begin{align*}
-\frac{\eta}6 \alpha_k^3 \|\d_k\|^3
&\leq f(\x_k + \alpha_k\d_k) - f(\x_k) \\
&\leq \alpha_k\nabla f_k^T\d_k + \frac{\alpha_k^2}{2} \d_k^T\nabla^2 f(\x_k) \d_k + \frac{L_H}6 \alpha_k^3 \|\d_k\|^3\\
& = \alpha_k\g_k^T\d_k + \frac{\alpha_k^2}{2} \d_k^T \H_k\d_k + \alpha_k(\nabla f_k-\g_k)^T\d_k + \frac{\alpha_k^2}{2} \d_k^T(\nabla^2 f(\x_k)-\H_k) \d_k \\
&~~~+ \frac{L_H}6 \alpha_k^3 \|\d_k\|^3 \\
& \leq -\frac{\alpha_k}2 (1-\zeta)\epsilonh\|\d_k\|^2 + \alpha_k\delta_{g,k}\|\d_k\| + \frac{\alpha_k^2}{2}\deltah\|\d_k\|^2 + \frac{L_H}6 \alpha_k^3\|\d_k\|^3  \quad \quad (\mbox{from \cref{eq:ic1}}) \\
& < \alpha_k\delta_{g,k}\|\d_k\|-\frac{\alpha_k}2 \|\d_k\|^2 \big((1-\zeta)\epsilonh-\deltah\big)  + \frac{L_H}6 \alpha_k^3\|\d_k\|^3 \quad \quad (\mbox{since $\alpha_k<1$}).
\end{align*}
By rearrangement, it follows that 
\begin{equation} \label{eq:sj8}
\frac{L_H+\eta}{6} \alpha_k^2 \| \d_k \|^2 - \frac12 \big((1-\zeta)\epsilonh-\deltah\big) \| \d_k \| + \delta_{g,k} >0.
\end{equation}
By substituting the definition of  $\alpha_k$ and using $\deltah\leq {(1-\zeta)}\epsilonh/{2}$ into the formula above, we have that \eqref{eq:sj8} implies
\begin{align*}
    \frac{L_H+\eta}{6}\left[ \frac{3(1 - \zeta)}{4(L_H + \eta)}\right] \frac{\epsilonh}{\|\d_k\|}\|\d_k\|^2 - \frac{(1-\zeta)\epsilonh}{4}\|\d_k\| + \delta_{g,k} & >0 \\
    \Leftrightarrow -\frac{(1-\zeta)}{8} \epsilonh \| \d_k \| + \delta_{g,k} & >0.
\end{align*}
By using $\delta_{g,k} \le (1-\zeta) \max\left(\epsilong, \min(\epsilonh\|\d_k\|, \|\g_k\|, \|\g_{k+1}\|)\right)/8$, this inequality implies that
\begin{equation} \label{eq:sj9}
- \epsilonh \|\d_k \| + \max\left(\epsilong, \min(\epsilonh\|\d_k\|, \|\g_k\|, \|\g_{k+1}\|)\right) >0.
\end{equation}
If $\epsilong > \min(\epsilonh\|\d_k\|, \|\g_k\|, \|\g_{k+1}\|)$, since $\epsilonh = \sqrt{L_H\epsilong}$, we have from \eqref{eq:sj9} that 
\[
   - \sqrt{L_H\epsilong}\|\d_k\| + \epsilong > 0  \Rightarrow \|\d_k\| < \sqrt{\epsilong/L_H},
\]
which contradicts our assumption $\|\d_k\| \geq \sqrt{{\epsilong}/{L_H}}=\epsilong/\epsilonh$.
Alternatively, if we assume that $\epsilong \le \min(\epsilonh\|\d_k\|, \|\g_k\|, \|\g_{k+1}\|)$, then from \eqref{eq:sj9}, it follows that  that
\[
0 < - \epsilonh\|\d_k\| + \min(\epsilonh\|\d_k\|, \|\g_k\|, \|\g_{k+1}\|)  \le  -\epsilonh\|\d_k\| +  \epsilonh\|\d_k\| =0,
\]
which is again a contradiction. Hence, our chosen value of $\alpha_k$ must satisfy \eqref{eq:bf1}, completing the~proof.
\end{proof}

Next, let us deal with the case when $\dtype=\NC$, which can be considered as a fixed-step alternative to Lemma~\ref{lemma:nc_from_cappedcg}. 
\begin{lemma}
\label{lemma:fixed_stepsize_nc_from_cappedcg}
Suppose that Assumption~\ref{ass:lip} is satisfied and that Condition~\ref{cond:opt_epsilong_epsilonh_fixedstepsize} holds for all $k$.
Suppose that at iteration $k$ of Algorithm~\ref{alg:fixed_stepsize}, we have $\|\g_k\| > \epsilong$, so that Procedure~\ref{alg:capped_cg} is called. 
When Procedure~\ref{alg:capped_cg} outputs a direction $\d_k$ with $\dtype=\NC$, we can choose the pre-defined step size 
\[
\alpha_k = \left( \frac{{(\|\d_k\|-\deltah)}/{2} + \sqrt{({(\|\d_k\|-\deltah)}/{2})^2 - 4{(L_H+\eta)}\delta_{g,k}/6}}{{(L_H+\eta)\|\d_k\|}/3} \right) \tilde\theta,
\]
% \steve{$6$ has been replaced by $3$ in the denominator, here and throughout.}
% \zhewei{Thank you so much for making the modification}
where $\tilde\theta$ is a parameter satisfying $(2-\sqrt{3})^2<\tilde\theta<1$. 
% \steve{There's a square in the lower bound now}
% \zhewei{I see, thanks for fixing it}
The resulting step $\x_{k+1}=\x_k+\alpha_k\d_k$ satisfies
$f(\x_k) -f(\x_{k+1}) \geq \cncb \epsilonh^3$,
where
\[
    \cncb :=  \frac{\eta}{6}\left[\frac{3 \tilde\theta}{4 (L_H+\eta)}\right]^3.                                                                                 
\]
\end{lemma}
\begin{proof} 
We start by noting that under the assumptions of the lemma, we have
\begin{equation} \label{nncc.1}
    \d_k^T \H_k \d_k \le -\epsilonh \| \d_k \|^2, \quad \| \d_k \| \ge \epsilonh,
\end{equation}
We replace the lower bound on $\| \d_k \|$ by the weaker bound $\| \d_k \| \ge \tfrac12 \epsilonh$ (so that we can reuse our results in the next lemma) to obtain
\begin{equation} \label{eq:lw2}
\| \d_k \| \ge \frac12 \epsilonh, \quad \delta_H \le \frac14 \epsilonh \le \frac12 \| \d_k \| \;\; \mbox{and so} \;\; \| \d_k \| - \delta_H \ge \frac12 \| \d_k \| \ge \frac14 \epsilonh.
\end{equation}
Note too that $\d_k^T \g_k \le 0$ by design, so that from Definition~\ref{cond:appr_gh} of $\delta_{g,k}$, we have
\begin{equation} \label{eq:es4}
\d_k^T \nabla f(\x_k) \le \d_k^T \g_k + \| \d_k \| \| \nabla f(\x_k) - \g_k \| \le \delta_{g,k} \| \d_k \|.
\end{equation}
We therefore have
\begin{align*}
    f(\x_k + \alpha_k\d_k) - \f(\x_k) 
&\leq \alpha_k\nabla f(\x_k)^T\d_k + \frac{\alpha_k^{2}}{2} \d_k^T\nabla^2 f(\x_k) \d_k + \frac{L_H}6 \alpha_k^{3} \|\d_k\|^3\\
&\leq \alpha_k\delta_{g,k}\|\d_k\| -\frac{\alpha_k^2}{2}\|\d_k\|^3 + \frac{\alpha_k^2}{2}\deltah\|\d_k\|^2  + \frac{L_H}6 \alpha_k^3 \|\d_k\|^3. \quad \quad (\mbox{from \cref{eq:rd9}})
\end{align*}
Thus condition \eqref{eq:bf1} will be satisfied provided that 
\[
\alpha_k\delta_{g,k}\|\d_k\| -\frac{\alpha_k^2}{2}\|\d_k\|^3 + \frac{\alpha_k^2}{2}\deltah\|\d_k\|^2  + \frac{L_H}6 \alpha_k^3 \|\d_k\|^3 \le -\frac{\eta}6 \alpha_k^{3} \|\d_k\|^3.
\]
By rearranging and dividing by $\alpha_k \| \d_k \|$, we find that $\alpha_k$ satisfies \eqref{eq:bf1} provided that the following quadratic inequality in $\alpha_k$ is satisfied:
\begin{equation} \label{eq:qe2}
\left( \frac{(L_H+\eta)\|\d_k\|^2}{6}\right) \alpha_k^2  - \left( \frac{\|\d_k\|(\|\d_k\|-\deltah)}{2} \right) \alpha_k + \delta_{g,k}  \le 0.
\end{equation}
In fact this inequality is satisfied provided that $\alpha_k \in [\beta_2,\beta_1]$, where
\begin{align*}
\beta_1 & := \frac{{(\|\d_k\|-\deltah)}/{2} + \sqrt{({(\|\d_k\|-\deltah)}/{2})^2 - 4{(L_H+\eta)}\delta_{g,k}/6}}{{(L_H+\eta)\|\d_k\|}/3},  \\
\beta_2 &:= \frac{{(\|\d_k\|-\deltah)}/{2} - \sqrt{({(\|\d_k\|-\deltah)}/{2})^2 - 4{(L_H+\eta)}\delta_{g,k}/6}}{{(L_H+\eta)\|\d_k\|}/3}. 
\end{align*}
To verify that the quantity under the square root is positive, we use \eqref{eq:lw2} to write
\begin{align*}
\left(\frac{\|\d_k\|-\deltah}{2} \right)^2 - 4\frac{(L_H+\eta)}6 \delta_{g, k} 
& \geq \frac{1}{16}\|\d_k\|^2 - \frac{2(L_H+\eta)}{3}\delta_{g, k}  \\
& \geq \frac{1}{64} \epsilonh^2 - \frac{2(L_H+\eta)}{3}\delta_{g, k} 
> 0,
\end{align*}
where the last inequality follows from Condition~\ref{cond:opt_epsilong_epsilonh_fixedstepsize}, since 
\[
\delta_{g,k} \le \frac{3}{2\times 65} \frac{\epsilonh^2}{L_H+\eta} < \frac{3}{128} \frac{\epsilonh^2}{L_H+\eta}.
\]
(Note that $0<\beta_2<\beta_1$.)

Next, we show that our choice of $\alpha_k$, which equals $\tilde\theta \beta_1$, lies in the interval $(\beta_2,\beta_1)$.
First, we have  $\alpha_k=\tilde\theta\beta_1 < \beta_1$ since $\tilde\theta<1$. 
Second, proving $\alpha_k > \beta_2$ is equivalent to showing that $\tilde\theta > \beta_2 / \beta_1$. Defining 
\[
z:= \frac{\| \d_k \|-\delta_H}{2}, \quad c:= \frac23 (L_H+\eta) \delta_{g,k},
\]
we see that 
\[
\beta_1 =  \frac{z+\sqrt{z^2-c}}{(L_H+\eta)\|\d_k \|/3}, \quad 
\beta_2 =  \frac{z-\sqrt{z^2-c}}{(L_H+\eta)\|\d_k \|/3}, 
\]
so that the required condition is 
\[
\tilde\theta > \beta_2 / \beta_1 = \frac{z-\sqrt{z^2-c}}{z+\sqrt{z^2-c}} = \frac{(z-\sqrt{z^2-c})^2}{c}.
\]
% \steve{This argument needed a lot of fixing.}
% \zhewei{Good catch! Thank you so much}
We have from \eqref{eq:lw2} and Condition~\ref{cond:opt_epsilong_epsilonh_fixedstepsize} that
\[
z^2 = \left( \frac{\|\d_k \| - \delta_H}{2} \right)^2 \ge \frac{\epsilonh^2}{64} > \frac{\epsilonh^2}{65} \ge \frac{8}{3} (L_H+\eta) \delta_{g,k} = 4c.
\]
Since $z-\sqrt{z^2-c}$ is a decreasing function of $z$ for all $z^2 >c>0$, we have by using $z^2>4c$ that
\[
\frac{\beta_2}{\beta_1} = \frac{(z-\sqrt{z^2-c})^2}{c} < \frac{(2 \sqrt{c} - \sqrt{4c-c})^2}{c} = (2-\sqrt{3})^2 < \tilde\theta.
\]
We have therefore proved that $\alpha_k \in [\beta_2,\beta_1]$, so that $\alpha_k$ satisfies \eqref{eq:bf1}.

From \eqref{eq:lw2}, we have
\begin{align} \nonumber
\alpha_k = \tilde\theta \beta_1
& = \tilde\theta\frac{{(\|\d_k\|-\deltah)}/{2} + \sqrt{({(\|\d_k\|-\deltah)}/{2})^2 - 4{(L_H+\eta)}\delta_{g,k}/6}}{{(L_H+\eta)\|\d_k\|}/3} \\
\label{eq:lw3}
& \geq \tilde\theta\frac{\|\d_k\|/4}{{(L_H+\eta)\|\d_k\|}/3} 
= \frac{3}{4}\frac{\tilde\theta}{L_H+\eta}.
\end{align}
The final claim of the theorem is obtained by substituting this lower bound on $\alpha_k$ into \eqref{eq:bf1}, and using $\|\d_k \| \ge \epsilonh$.
\end{proof}

% \steve{Changed the next lemma to deal with Procedure 2 only. Changed the statement. The proof of the previous lemma has been retooled to make use only of the weaker bound available in this one.}
% \zhewei{Great idea to make the previous lemma loose!}

The next lemma shows that when $\dtype=\NC$ is obtained from Procedure~\ref{alg:minimum_eigenvalue}, the same fixed step size as in \cref{lemma:fixed_stepsize_nc_from_cappedcg} can be used, with the same lower bound on improvement in $f$.
\begin{lemma}\label{lemma:fixedstepsize_nc_proc2}
Suppose that Assumption~\ref{ass:lip} is satisfied and that Condition~\ref{cond:opt_epsilong_epsilonh_fixedstepsize} holds for all $k$.
Suppose that at iteration $k$ of Algorithm~\ref{alg:fixed_stepsize}, the step $\d_k$ is of negative curvature type, obtained from Procedure~\ref{alg:minimum_eigenvalue}. Then when we define $\alpha_k$ as in \cref{lemma:fixed_stepsize_nc_from_cappedcg}, we obtain
\begin{equation}
f(\x_k) - f(\x_k + \alpha_k \d_k) \ge \frac18 \cncb \epsilon_H^3, 
\end{equation}
where $\cncb$ is defined in \cref{lemma:fixed_stepsize_nc_from_cappedcg}.
\end{lemma}
\begin{proof}
Note that for $\d_k$ obtained from Procedure~\ref{alg:minimum_eigenvalue}, we have
\[
\d_k^T\H\d_k \leq -\frac{1}{2}\epsilonh\|\d_k\|^2, \quad \| \d_k \| \ge \frac12 \epsilong.
\]
Since the bulk of the proof of \cref{lemma:fixed_stepsize_nc_from_cappedcg} uses only the latter lower bound on $\|\d_k\|$, we can use this proof to derive the same lower bound \eqref{eq:lw3} on $\alpha_k$. The result follows by substituting this lower bound together with $\|\d_k \| \ge \epsilonh/2$ into \eqref{eq:bf1}.
\end{proof}

Using \cref{lemma:fixed_stepsize_with_sol,lemma:fixed_stepsize_nc_from_cappedcg,lemma:fixedstepsize_nc_proc2,lemma:terminate_step_condition_withlinesearch}, we are now ready to give the iteration complexity of \cref{alg:fixed_stepsize}.
\begin{theorem}
\label{thm:fixed_stepsize_complexity}
Suppose that Assumption~\ref{ass:lip} is satisfied and that Condition~\ref{cond:opt_epsilong_epsilonh_fixedstepsize} holds for all $k$.
For a given $\epsilon>0$, let $\epsilonh = \sqrt{L_H \epsilon}, \epsilong = \epsilon$. 
Define 
\begin{equation}\label{eq:k2}
\bar K_2 := 2 \left\lceil \frac{f(\x_0) - \flow}{\min\{\csolb,\cncb/8\} L_H^{3/2}} \epsilon^{-3/2} \right \rceil + 3,
\end{equation}
where $\csolb$ and $\cncb$ are defined in \cref{lemma:fixed_stepsize_with_sol} and \cref{lemma:fixed_stepsize_nc_from_cappedcg}, respectively. 
Then \cref{alg:fixed_stepsize} terminates in at most $\bar K_2$ iterations at a point $\x$ satisfying 
$\norm{\nabla f(\x)} \lesssim \epsilon$.
Moreover, with probability at least $(1 - \delta)^{\bar K_2}$, the point returned by \cref{alg:fixed_stepsize} also satisfies the approximate second-order condition
%\begin{equation} \label{eqn:2oeps}
$\lambda_{\min}(\nabla^2 f(\x)) \gtrsim -\sqrt{L_H\epsilon}$.
% \end{equation}
\reftwo{Here again, $\lesssim$ and $\gtrsim$ denote that the corresponding inequality holds up to a certain constant that is independent of $ \epsilon $ and $ L_{H} $.}
\end{theorem}
\begin{proof}
The proof tracks that of \cref{thm:opt_iteration_comp} closely, so we omit much of the detail and discussion.

For contradiction, we assume that \cref{alg:fixed_stepsize} runs for at least $K$ steps, where $K>\bar K_2$. We partition the set of iteration indices $\{1,2,\dotsc,K \}$ into the same sets $\mathcal K_1, \dotsc, \mathcal K_5$ as in the proof of \cref{thm:opt_iteration_comp}. 
Considering each of these sets in turn, we have the following.

\paragraph{Case 1:} $k \in \mathcal K_1$. Either \cref{alg:fixed_stepsize} terminates (which happens at most once for $k \in \mathcal K_1$) or we achieve a reduction in $f$ of at least $\tfrac18 \cncb \epsilonh^3 = \tfrac18 \cncb L_H^{3/2} \epsilon^{3/2}$ (\cref{lemma:fixedstepsize_nc_proc2}).

\paragraph{Cases 2 and 3:} $k \in \mathcal K_2 \cup \mathcal K_3$. $f$ is reduced by at least $\csolb L_H^{3/2} \epsilon^{3/2}$ (\cref{lemma:fixed_stepsize_with_sol}).

\paragraph{Case 4:} $k \in \mathcal K_4$. The algorithm terminates, so we must have $|\mathcal K_4| \le 1$.

\paragraph{Case 5:} $k \in \mathcal K_5$. Either the algorithm terminates, or we achieve a reduction of at least $\cncb L_H^{3/2} \epsilon^{3/2}$ (\cref{lemma:fixed_stepsize_nc_from_cappedcg}).

Reasoning as in the proof of \cref{thm:opt_iteration_comp}, we have that 
\[
    f(\x_0) - \flow \ge (| \mathcal K_1|-1) \tfrac18 \cncb L_H^{3/2} \epsilon^{3/2} + (| \mathcal K_2| + | \mathcal K_3| ) \csolb L_H^{3/2} \epsilon^{3/2}  + (| \mathcal K_5|-1) \cncb L_H^{3/2} \epsilon^{3/2},
\]
from which we obtain
\begin{align*}
    | \mathcal K_1| + | \mathcal K_5| -2 & \le \frac{f(\x_0)-\flow}{\tfrac18 \cncb L_H^{3/2}} \epsilon^{-3/2}, \\
    | \mathcal K_2| + | \mathcal K_3|  & \le \frac{f(\x_0)-\flow}{\csolb L_H^{3/2}}\epsilon^{-3/2}.
\end{align*}
By using these bounds along with $| \mathcal K_4| \le 1$, we obtain
\[
K \le \sum_{i=1}^5 | \mathcal K_i|  \le 2 \frac{f(\x_0) - \flow}{\min(\cncb/8,\csolb) L_H^{3/2}} \epsilon^{-3/2} + 3,
\]
which contradicts  our assumption that $K > \bar K_2$.

The proof of the remaining claim, concerning the approximate second-order condition, is identical to the corresponding section in the proof of \cref{thm:opt_iteration_comp}.

\end{proof}

Note that the worst-case iteration complexity of \cref{alg:fixed_stepsize} has the same dependence on $\epsilon$ as \cref{alg:inexact_withlinesearch} despite the function evaluation no longer being required. The terms in the bound that do not depend on $\epsilon$ are, however, generally worse for \cref{alg:fixed_stepsize}.

\reftwo{We conclude with a discussion of Conditions \ref{cond:opt_epsilong_epsilonh} and \ref{cond:opt_epsilong_epsilonh_fixedstepsize}.
These conditions allow for the accuracy of $\g_k$ to be chosen adaptively, depending on problem-dependent constants, algorithmic parameters, the desired solution tolerances $\epsilong$ and $\epsilonh$, and the quantities $\|\d_k\|$, $\|g_{k+1}\|$, and $\|\g_k\|$. The quantity $\|\g_k\|$ is easy to evaluate (since, after all, $\g_k$ is the quantity actually calculated). However, the dependence on the quantities $\| \d_k\|$ and $\| \g_{k+1} \|$ is more problematic, since $\g_k$ is needed to compute both $\d_k$ and $\g_{k+1}$. 
Thus, the bounds on $\delta_{g,k}$ in Conditions \ref{cond:opt_epsilong_epsilonh} and \ref{cond:opt_epsilong_epsilonh_fixedstepsize} can be checked only ``in retrospect," not enforced as an a priori condition.
We can deal with this issue by checking the bound on $\delta_{g,k}$ after the step to $\x_{k+1}$ has been taken. 
if it fails to be satisfied, we can improve the accuracy of $\g_k$ and re-do iteration $k$.
If we halve $\delta_{g,k}$ each time the step is recomputed, the number of recomputations is at worst a multiple of $\log \epsilong$ (since the bound on $\delta_{g,k}$ in both conditions is at least $(1-\zeta)\epsilong/8$), so our complexity bounds are not affected significantly.
We choose to elide this fairly uninteresting issue in our analysis, and simply assume for simplicity that the  relevant bound on $\delta_{g,k}$ holds at each iteration.
}

% {\bf PREVIOUS TEXT:}
\iffalse
\reftwo{Before ending this section, we emphasize that although the adaptivity of Conditions \ref{cond:opt_epsilong_epsilonh} and \ref{cond:opt_epsilong_epsilonh_fixedstepsize} is a desirable theoretical property, the main drawback lies in the difficulty of enforcing them in practice. 
Indeed, a priori enforcing Conditions \ref{cond:opt_epsilong_epsilonh} and \ref{cond:opt_epsilong_epsilonh_fixedstepsize} requires one to have already taken the $k^{\textnormal{th}}$ iteration, which itself can be done only after computing $\g_{k}$, hence creating a vicious circle. 
A posteriori guarantees can be given if one obtains a lower-bound estimate on the yet-to-be-computed $ \|\g_{k}\| $ and $ \|\g_{k+1}\| $, i.e., to have  $ g_{0}>0 $  such that $ g_{0} \leq \min\{\|\g_{k}\|,\|\g_{k+1}\|\} $. 
This allows one to consider a stronger, but practically enforceable, condition. However, to obtain such a lower-bound estimate on $\|\g_{k}\|$ and $\|\g_{k+1}\|$, one has to resort to a recursive procedure, which necessitates repeated constructions of the approximate gradient and subsequent solutions of the corresponding sub-problems. 
Clearly, this procedure will result in a significant computational overhead and will lead to undesirable theoretical complexities. 
An important area for future work is the investigation of ways to modify the adaptivity of Conditions~\ref{cond:opt_epsilong_epsilonh} and \ref{cond:opt_epsilong_epsilonh_fixedstepsize} in such a way that they can be enforced in practice.}
\fi

\subsection{Evaluation complexity of Algorithm~\ref{alg:fixed_stepsize} for finite-sum problems}
\label{sec:finite-sum}
When $f$ has finite-sum form \cref{eq:finte_sum_problem} for $n \gg 1$, we consider subsampling schemes for estimating $\g_k$ and $\H_k$, as in \cite{roosta2019sub,xuNonconvexTheoretical2017}. 
We can define the subsampled quantities as follows
\begin{align}\label{eq:sub_gh}
\g \triangleq \frac{1}{\Abs{\mathcal S_g}} \sum_{i\in\mathcal S_g} \nabla f_i(\x), ~~ \text{and} ~~~ \H \triangleq \frac{1}{\Abs{\mathcal S_H}} \sum_{i\in\mathcal S_H} \nabla^2 f_i(\x),
\end{align}
where $\mathcal S_g, \mathcal S_H\subset\{1,\cdots, n\}$ are the subsample batches for the estimates of the gradient and Hessian, respectively. 
In~\citet[Lemma 1 and 2]{roosta2019sub} and \citet[Lemma 16]{xuNonconvexTheoretical2017}, 
it is shown that with a uniform sampling strategy, the following lemma can be proved.
\begin{lemma}[Sampling complexity \citep{roosta2019sub,xuNonconvexTheoretical2017}]
\label{lemma:sampling}
Suppose that Assumption~\ref{ass:lip} is satisfied, and let $\bar\delta \in (0,1)$ be given.
Suppose that at iteration $k$ of Algorithm~\ref{alg:fixed_stepsize}, $\delta_{g,k}$ and $\delta_H$ are as defined in Condition~\ref{cond:opt_epsilong_epsilonh_fixedstepsize}.
Also, let $ 0< K_g,K_H < \infty $ be such that $ \norm{\nabla f_i(\x)} \le K_g$ and $\norm{\nabla^2 f_i(\x)} \le K_H$ for all $\x$ belonging to the set defined in Assumption~\ref{ass:lip}.
For $\g_k$ and $\H_k$ defined as in \eqref{eq:sub_gh} with $\x=\x_k$, and subsample sets  $\mathcal S_g = \mathcal S_{g,k}$ and $\mathcal S_H$ satisfying 
\begin{align*}
\Abs{\mathcal S_{g,k}} \ge \frac{16K_g^2}{\delta_{g,k}^2} \log\frac{1}{\bar\delta} \quad \text{and} \quad \Abs{\mathcal S_H} \ge \frac{16K_H^2}{\delta_H^2} \log\frac{2d}{\bar\delta},
\end{align*}
we have with probability at least $1-\bar\delta$ that Condition~\ref{cond:opt_epsilong_epsilonh_fixedstepsize} holds for the given values of $\delta_{g,k}$ and $\delta_H$.
\end{lemma}

For the choices of $\epsilon_g$ and $\epsilon_H$ being used in this section, and assuming that $\delta_{g,k}$ and $\delta_{H}$ are set to their upper bounds in Condition~\ref{cond:opt_epsilong_epsilonh_fixedstepsize}, we can derive a uniform condition on the required subsample~sizes.
\begin{lemma} \label{lemma:sampling2}
Suppose the conditions of Lemma~\ref{lemma:sampling} holds, and that for some $\epsilon>0$, we set $\epsilon_H = \sqrt{L_H \epsilon}$ and $\epsilon_g = \epsilon$.  
Suppose that at some iteration $k$, $\delta_{g,k}$ and $\delta_H$ are set to their upper bounds in Condition~\ref{cond:opt_epsilong_epsilonh_fixedstepsize}. Then we have that $\delta_{g,k} \ge \bar\delta_g$ for all $k$ and $\delta_H = \bar\delta_H$, where
\begin{equation} \label{eq:bardgdh}
\reftwo{\bar\delta_g = \frac{1-\zeta}{8} \min \left( \frac{3 L_H \epsilon }{65(L_H+\eta)}, \epsilon \right) = \bigO(\epsilon),} \quad
\bar\delta_H = \left( \frac{1-\zeta}{4} \right) \sqrt{L_H \epsilon} = \bigO(\epsilon^{1/2}).
\end{equation}
Moreover, when $\g_k$ and $\H_k$ are estimated from \eqref{eq:sub_gh} with $\x=\x_k$ and subsample sets $\mathcal S_g = \mathcal S_{g,k}$ and $\mathcal S_H$ satisfying
\begin{align*}
\Abs{\mathcal S_g} \ge \frac{16K_g^2}{\bar\delta_g^2} \log\frac{1}{\bar\delta} = \bigO(\epsilon^{-2}), \quad \Abs{\mathcal S_H} \ge \frac{16K_H^2}{\bar\delta_H^2} \log\frac{2d}{\bar\delta} = \bigO(\epsilon^{-1}),
\end{align*}
then  Condition~\ref{cond:opt_epsilong_epsilonh_fixedstepsize} is satisfied at iteration $k$ with probability at least $1-\bar\delta$.
\end{lemma}
\begin{proof}
The right-hand side of the bound on $\delta_{g,k}$ in Condition~\ref{cond:opt_epsilong_epsilonh_fixedstepsize} is bounded below by
\[
\frac{1-\zeta}{8} \min \left( \frac{3 \epsilon_H^2}{65(L_H+\eta)}, \epsilon_g \right) =
\frac{1-\zeta}{8} \min \left( \frac{3 L_H \epsilon }{65(L_H+\eta)}, \epsilon \right)  = \bar\delta_g = \bigO(\epsilon),
\]
as claimed.
The claims concerning $\bar\delta_H$ are immediate.
\end{proof}

By combining \cref{lemma:sampling2} with Theorem~\ref{thm:fixed_stepsize_complexity}, we can obtain an oracle complexity result in which the oracle is either an evaluation of a gradient $\nabla f_i$ for some $i=1,2,\dotsc,n$ or a Hessian-vector product of the form $\nabla^2 f_i(\x) \v$, for some $i=1,2,\dotsc,n$ and some $\x, \v \in \R^d$. 
The result is complicated by the fact that there is a probability of failure to satisfy Condition~\ref{cond:opt_epsilong_epsilonh_fixedstepsize} at each $k$, to go along with the possible failure, noted in the previous section, to detect negative curvature when Procedure~\ref{alg:minimum_eigenvalue} is invoked.
For our result below, we consider the case in which failure to satisfy Condition~\ref{cond:opt_epsilong_epsilonh_fixedstepsize} never occurs at any iteration, Since there are at most $\bar{K}_2$ iterations, this case occurs with probability at least $(1-\bar\delta)^{\bar{K}_2}$.

\begin{corollary}[Evaluation Complexity of Algorithm~\ref{alg:fixed_stepsize} for finite-sum problem \cref{eq:finte_sum_problem}]
\label{cor:tr_prob}
Suppose that Assumption~\ref{ass:lip} is satisfied.
Let $\bar\delta \in (0,1)$ be given, and suppose that at each iteration $k$, $\g_k$ and $\H_k$ are obtained from \eqref{eq:sub_gh}, with $\mathcal S_g = \mathcal S_{g,k}$ and $\mathcal S_H$ satisfying the lower bounds in Lemma~\ref{lemma:sampling}, where $\delta_{g,k} \ge \bar\delta_g$ and $\delta_H \ge \bar\delta_H$, with $\bar\delta_g$ and $\bar\delta_H$ defined in \eqref{eq:bardgdh}.
For a given $\epsilon>0$, let $\epsilonh = \sqrt{L_H \epsilon}, \epsilong = \epsilon$. 
Let $\bar{K}_2$ be defined as in \eqref{eq:k2}.
Then with probability at least $(1-\bar\delta)^{\bar{K}_2}(1-\delta)^{\bar{K}_2}$, Algorithm~\ref{alg:fixed_stepsize} terminates in at most $\bar K_2$ iterations at a point $\x$ satisfying 
$\norm{\nabla f(\x)} \lesssim \epsilon$ and $\lambda_{\min}(\nabla^2 f(\x)) \gtrsim -\sqrt{L_H\epsilon}$.
\reftwo{Again, $\lesssim$ and $\gtrsim$ denote that the corresponding inequality holds up to a certain constant that is independent of $ \epsilon $ and $ L_{H} $.}
Moreover, the total number of oracle calls is bounded by 
	\begin{align*}
	   & \underbrace{\left(2\left\lceil \frac{(f(\x_0) - \flow)}{\min\{\csolb,\cncb/8\}} \epsilon^{-3/2} \right \rceil + 3 \right)}_{\bar K_2} \cdot \left(\underbrace{\frac{16K_g^2}{\bar\delta_{g}^2} \log\frac{1}{\bar\delta}}_{\text{Gradient Sampling}} +  \underbrace{\frac{16K_H^2}{\bar\delta_H^2} \log\frac{2d}{\bar\delta}}_{\text{Hessian Sampling}}\cdot \left(\underbrace{\tilde\OM(\epsilon^{-1/4})}_{Procedure~\ref{alg:capped_cg}} + \underbrace{\bigO(\epsilon^{-1/4})}_{Procedure~\ref{alg:minimum_eigenvalue}}  \right) \right) \\
	    & = \bigO(\epsilon^{-3/2}) \cdot(\OM(\epsilon^{-2}) + \tilde\OM(\epsilon^{-1} \times \epsilon^{-1/4})) \\
	    & = \OM(\epsilon^{-7/2}).
	\end{align*}
\end{corollary}

\refone{
As mentioned earlier, \cref{alg:fixed_stepsize} requires knowledge of an upper bound of the Lipschitz contstant $L_H$ of the Hessian matrix. 
In addition, the sample complexity derived in \cref{cor:tr_prob} depends on upper estimates of $K_g$ and $K_H$, which may be unavailable for many non-convex problems. Fortunately, for many non-convex objectives of interest in machine learning and statistical analysis, we can readily obtain reasonable estimates of these quantities. 
\cref{tab:L_H_bound} provides estimates on $L_H$ for some examples of such objectives. See \cref{tab:K_H_K_g_bound} for upper bounds on $K_g$ and $K_H$ for such problems. 
Equipped with these estimates, we can give a more refined complexity analysis tailored for the problems in \cref{tab:L_H_bound,tab:K_H_K_g_bound}. 
Indeed, since for the constants $\csolb$ and $\cncb$ in \cref{lemma:fixed_stepsize_nc_from_cappedcg,lemma:fixed_stepsize_with_sol}, we have
$\csolb \in \Omega(1/L_H^3)$, $\cncb \in \Omega(1/L_H^3)$, from \cref{tab:L_H_bound,tab:K_H_K_g_bound,cor:tr_prob}, it follows that the total number of oracle calls for these problems is at most
\begin{align*}
	& \tilde \OM \left[ \left( \left( \max_{i} \|\a_{i}\|^{9} \right) (f(\x_0) - \flow) \epsilon^{-3/2} \right) \right] \cdot \left( \tilde \OM \left( \left( \max_{i} \|\a_{i}\|^{2} \right) \epsilon^{-2} \right) + \tilde \OM \left( \left( \max_{i} \|\a_{i}\| \right) \epsilon^{-5/4}\right) \right) \\
	& = \tilde \OM\left(\epsilon^{-7/2} (f(\x_0) - \flow) \max_{i} \left\{1,\|\a_{i}\|\right\}^{11} \right),
\end{align*}
where for simplicity we have assumed $ |b_{i}| \leq 1 $, e.g., binary classification problems.

\begin{table}[!htb]
	\centering
	\refone{\caption{The upper bound of $K_g$ and $K_H$ for the non-convex finite-sum minimization problems of \cref{tab:L_H_bound}. \label{tab:K_H_K_g_bound}}}
	\centering
	\begin{adjustbox}{width=1\linewidth}
		\refone{\begin{tabular}{lcccc}
			\toprule
			Problem Formulation & Predictor Function & \multicolumn{1}{m{6cm}}{\centering Upper bound of $K_g$} & \multicolumn{1}{m{5cm}}{\centering Upper bound of $K_H$} \\
			\midrule
			$\sum\limits_{i=1}^n(b_i-\phi(\langle \a_i, \x\rangle))^2$ & $\phi(z) =1/{(1+e^{-z})}$
			& \multicolumn{1}{m{6cm}}{\centering$\displaystyle \max_{i=1,\ldots,n} \; \left(|b_i| + 1\right)\|\a_{i}\|/2$} & $\displaystyle \max_{i=1,\ldots,n} \; \left(|b_i| + 2\right)\|\a_{i}\|^{2}$ \\ \\
			$\sum\limits_{i=1}^n(b_i-\phi(\langle \a_i, \x\rangle))^2$ & $\phi(z) ={(e^{z}-e^{-z})}/{(e^{z}+e^{-z})}$
			& \multicolumn{1}{m{6cm}}{\centering$\displaystyle \max_{i=1,\ldots,n} \; 2 \left(|b_i| + 1\right)\|\a_{i}\|$} & $\displaystyle \max_{i=1,\ldots,n} \; \left(|b_i| + 2\right)\|\a_{i}\|^{2}$ \\\\ 
			$\sum\limits_{i=1}^n\phi(b_i - \langle \a_i, \x\rangle)$ & $\phi(z) ={(1 -e^{-\alpha z^2})}/{\alpha}$ & \multicolumn{1}{m{6cm}}{\centering$\displaystyle \sqrt{2/\alpha}\max_{i=1,\ldots,n} \; \|\a_{i}\|$} & $\displaystyle 2 \max_{i=1,\ldots,n} \;\|\a_{i}\|^{2}$\\ 
			\bottomrule
		\end{tabular}}
	\end{adjustbox}
\end{table}

}

\section{Numerical evaluation}
\label{sec:numerical_experiments}

In this section, we evaluate the performance of \cref{alg:inexact_withlinesearch,alg:fixed_stepsize} on three model problems in the form of finite-sum minimization: nonlinear least squares (NLS), multilayer perceptron (MLP), and variational autoencoder (VAE). 
Our aim here is to illustrate the efficiency gained from gradient and Hessian approximations as compared with the exact counterpart in \cite{royer2018newton}.
More specifically, in our numerical examples, we consider the following algorithms.
\begin{itemize}[leftmargin=*,wide=0em]%, itemsep=-3pt,topsep=-3pt]
\item {\texttt {Full NTCG}}: Newton Method with Capped-CG solver with full gradient and Hessian evaluations, as developed in \cite{royer2018newton}. 
\item {\texttt {SubH NTCG}} ({\bf this work}): Variant of \cite{royer2018newton} where Hessian is approximated. We consider this setting as an intermediary between the full algorithm and those where both the gradient and the Hessian are approximated. Sample sizes for approximating Hessian for experiments using NLS, MLP, and VAE, are $ 0.01n $, $ 0.02n $, and $ 0.02n $, respectively. 
\item {\texttt {Inexact NTCG Full-Eval}} ({\bf this work}): Newton Method with Capped-CG solver with back-tracking line-search where both the gradient and the Hessian are approximated. 
To perform the backtracking line search, we employ the full dataset to evaluate the objective function. 
% {\bf NEED TO CHANGE THIS SINCE $ \sgn(\d^T\nabla f(\x_k)) $ IS NO LONGER NEEDED EXACTLY IN \cref{alg:inexact_withlinesearch}.} \refboth{Following \cref{alg:inexact_withlinesearch}, the gradients used for the steps involving the ``sgn'' function and the inner-products, i.e., $ \sgn(\d^T\nabla f(\x_k)) $, are also evaluated exactly.} 
The sample size for estimating the gradient is adaptively calculated as follows: if $\|\g_{t}\| \geq 1.2 \|\g_{t-1}\|$ or $\|\g_{t}\| \leq \|\g_{t-1}\|/1.2$, then the sample size is decreased or increased, respectively, by a factor of 1.2. 
Otherwise, we maintain the same sample size as the previous iteration. \reftwo{The initial sample size to approximate the gradient for the experiments of \cref{sec:nls} is set to $0.05 n$, while for the experiments of \cref{sec:mlp,sec:vae}, we use an initial sample size of 10,000.}
The sample size for approximating the Hessian is set the same as that in \texttt{SubH NTCG}. 
\item {\texttt {Inexact NTCG Fixed}} ({\bf this work}): Newton Method with Capped-CG solver, using approximations of both the gradient and the Hessian and fixed step-sizes. 
The step sizes are predefined as follows: for NLS experiments, we use $\alpha_k=0.04$ for $\dtype=\NC$ and $\alpha_k=0.2$ for $\dtype=\SOL$, while for simulations on MLP/VAE models, we consider $\alpha_k=0.1$ for $\dtype=\NC$ and $\alpha_k=\sqrt{0.1}$ for $\dtype=\SOL$. 
The gradient and Hessian approximations are done as in the previous two variants.
\item {\texttt {Inexact NTCG Sub-Eval}}: This method is almost identical to \texttt {Inexact NTCG Full-Eval}, however, 
% there are a few critical differences. \refboth{The approximate gradient used throughout the algorithm, in particualr the quantities $ \sgn(\d^T\nabla f(\x_k)) $ are approximated by using the gradient estimate as $ \sgn(\d^T \g_k) $}. Furthermore, 
the backtracking line search is performed on estimates of the objective function using the same samples as the ones used in gradient approximation. 
Of course, our theoretical analysis does not immediately support this variant. However, we have found this strategy to be highly effective in practice, and we intend to theoretically investigate it in future work. 
\end{itemize}

\reftwo{In all of our experiments, we run each stochastic method five times (starting from the same initial point), and plot the average run (solid line) and 1-standard deviation band (shaded regions). 
To avoid cluttering the plots, we only show the upper deviation from the average, since the lower deviation band is almost identical on all of our experiments.}

\refone{We note that the step-size implies by \cref{alg:fixed_stepsize} is very pessimistic and hence small. This is a byproduct of our worst-case analysis, which comprises of descent obtained from a sequence of conservative steps. Requiring small step-lengths to provide a convergence guarantee is perhaps the main drawback for the worst-case style of analysis, which is almost ubiquitous within the optimization literature, e.g., fixed step-size of length $1/L_g$ for gradient descent on smooth unconstrained problems. 
Our numerical example shows that much larger step-sizes than those prescribed by \cref{alg:fixed_stepsize} can be employed in practice. 
We suspect this to be the case for most practical applications.}

\reftwo{
Although in Algorithms \ref{alg:inexact_withlinesearch} and \ref{alg:fixed_stepsize}, the case where $ \|\d_{k}\| $ is small (relative to the ratio $\epsilon_{g}/\epsilon_{H}$) is crucial in obtaining theoretical guarantees, in all of our simulations, we have found that performing line search directly with such small $ \d_{k} $ and without resorting to Procedure \ref{alg:minimum_eigenvalue} in fact yields reasonable progress. In this light, in all of our implementations, we have made the practical decision to omit Lines 9-16 of Algorithms \ref{alg:inexact_withlinesearch} and \ref{alg:fixed_stepsize}.}

%As we mentioned in Section~\ref{sec:main_result}, we pick the sign of the direction $\d_t$ in relation to $\g_k$. 
%Also, we do not verify the condition of Lemma~\ref{lemma:terminate_step_condition_withlinesearch}. 
Similar to \cite{xuNonconvexEmpirical2017,yao2018inexact}, the performance of all the algorithms is measured by tallying the total \textit{number of propagations}, that is, the number of oracle calls of function, gradient, and Hessian-vector products. This is so since comparing algorithms in terms of ``wall-clock'' time can be highly affected by their particular implementation details as well as system specifications. In contrast, counting the number of oracle calls, as an implementation and system independent unit of complexity, is most appropriate and fair. 
More specifically, after computing $ f_{i}(\x) $, which accounts for one oracle call, computing the corresponding gradient $ \nabla f_{i}(\x) $ is equivalent to one additional function evaluation, i.e., two oracle calls are needed to compute $ \nabla f_{i}(\x) $. Our implementations are Hessian-free, i.e., we merely require Hessian-vector products instead of using the explicit Hessian. For this, each Hessian-vector product $ \nabla^{2} f_{i}(\x) \vv $ amounts to two additional function evaluations, as compared with gradient evaluation, i.e., four oracle calls are used to evaluate $ \nabla^{2} f_{i}(\x) \vv$.

\subsection{Nonlinear least squares}
\label{sec:nls}
We first consider the simple, yet illustrative, non-linear least squares problems arising from the task of binary classification with squared loss.\footnote{Logistic loss, the ``standard'' loss used in this task, leads to a convex objective. We use squared loss to obtain a nonconvex objective.}
Given training data $\{\a_i,b_i\}_{i=1}^n$, where $\a_i\in\bbR^d, b_i\in\{0,1\}$, we solve the empirical risk minimization problem
\begin{align*}
\min_{\x\in\bbR^d} \frac1n \sum_{i=1}^n \Big(b_i-\phi\big(\langle\a_i,\x\rangle\big)\Big)^2,
\end{align*}
where $\phi(z)$ is the sigmoid function: $\phi(z) = 1/(1+e^{-z})$.
Datasets are taken from \texttt{LIBSVM} library~\citep{libsvm}; see \cref{tab:data1} for details. 
We use the same setup as in~\cite{yao2018inexact}. 
% \begin{table}[htb]
% \caption{Datasets used for experiments with non-linear least squares.}
% \label{tab:data1}
% \centering
% \begin{tabular}{cccc}
% \toprule
% \sc Data & $n$ & $d$ \\ 
% \midrule
% {\texttt{ covertype}} & 464,810 & 54
% \\
% {\texttt{ ijcnn1}} & 49,990 & 22 
% \\
% \bottomrule
% \end{tabular}
% \end{table}

\begin{table}[!htb]
		\centering
		\caption{Datasets used for NLP experiments.}
		\label{tab:data1}
		\begin{tabular}{lccccc} 
			\toprule
			\sc Data & $n$ & $d$ \\ 
			\midrule
			{\texttt{ covertype}} & 464,810 & 54
			\\
			{\texttt{ ijcnn1}} & 49,990 & 22 
			\\
			\bottomrule
		\end{tabular}
\end{table}

The comparison between different NTCG algorithms is shown in Figure~\ref{fig:nls_result_ntcg}. 
\reftwo{It is clear that, for a given value of the loss, all inexact variants in the \texttt{Inexact NTCG} family converge faster, i.e., with fewer oracle calls.} 
Clearly, lower per-iteration cost of \texttt{Inexact NTCG Fixed} comes at the cost of slower overall convergence as compared with \texttt{Inexact NTCG Sub-Eval}. This is mainly because the step size obtained as part of the line-search procedure can generally result in a better decrease in function value.
\reftwo{For this problem we could refer to \cref{tab:L_H_bound} and explicitly compute the fixed step-size prescribed by \cref{alg:fixed_stepsize}. 
As mentioned earlier, the resulting step size is overly conservative. 
% We realize that this conservative estimate is a mere by-product of our proof. 
Our simulations show that much larger step-sizes yield convergent algorithms. In this light, our fixed step-sizes are chosen  without regard to the value prescribed  in \cref{alg:fixed_stepsize}, but are based rather on numerical experience.}

%%%%%%%%%%%%%%%%%%%%%%%%%%%%%%%%%%%%%%%%%%%%%%%%%%%%%%%%%%%%%%%%%%%%%%%%%%%%%%%%%%%
\begin{figure}[tbp]
\begin{center}
  \includegraphics[width=.49\textwidth]{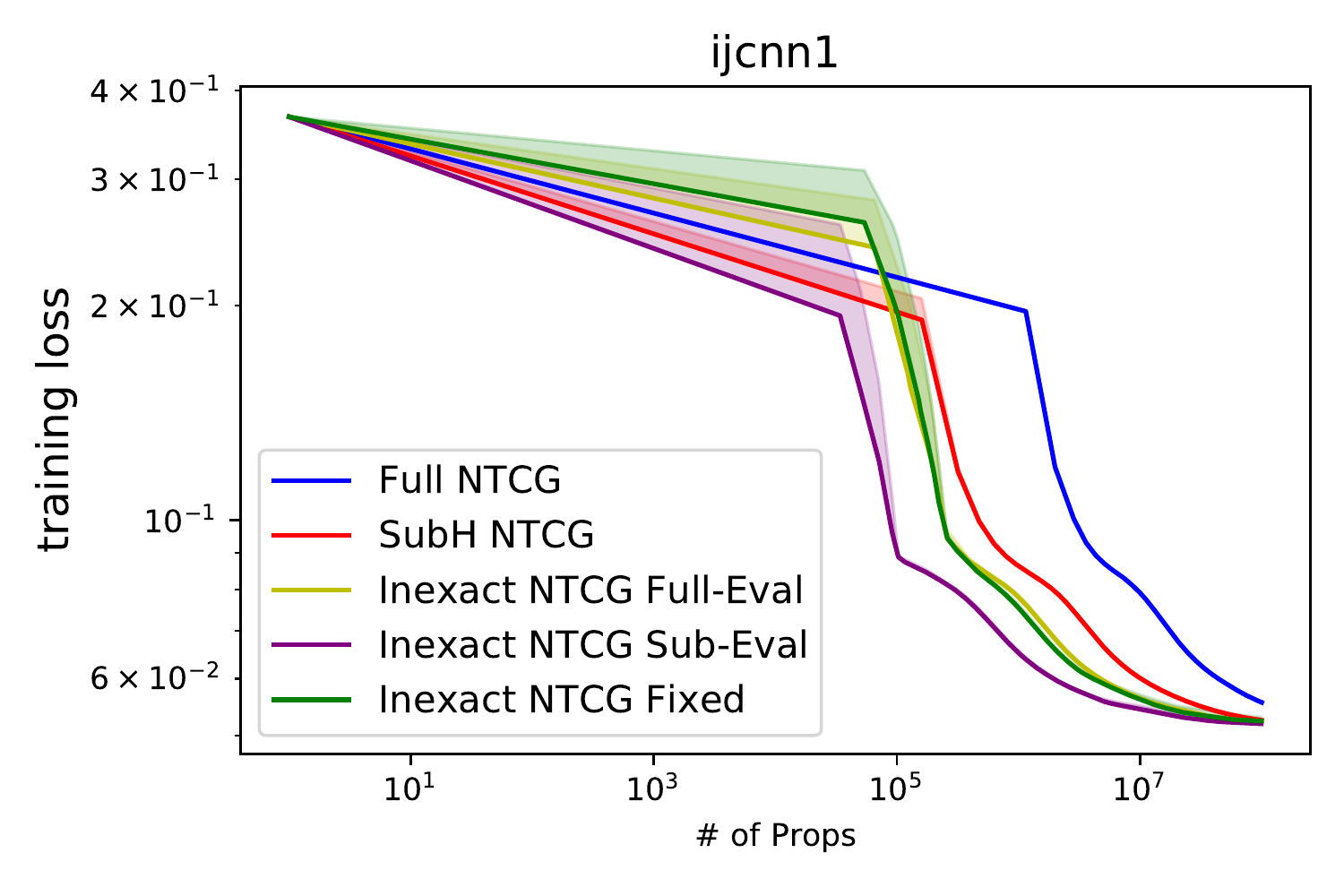}
  \includegraphics[width=.49\textwidth]{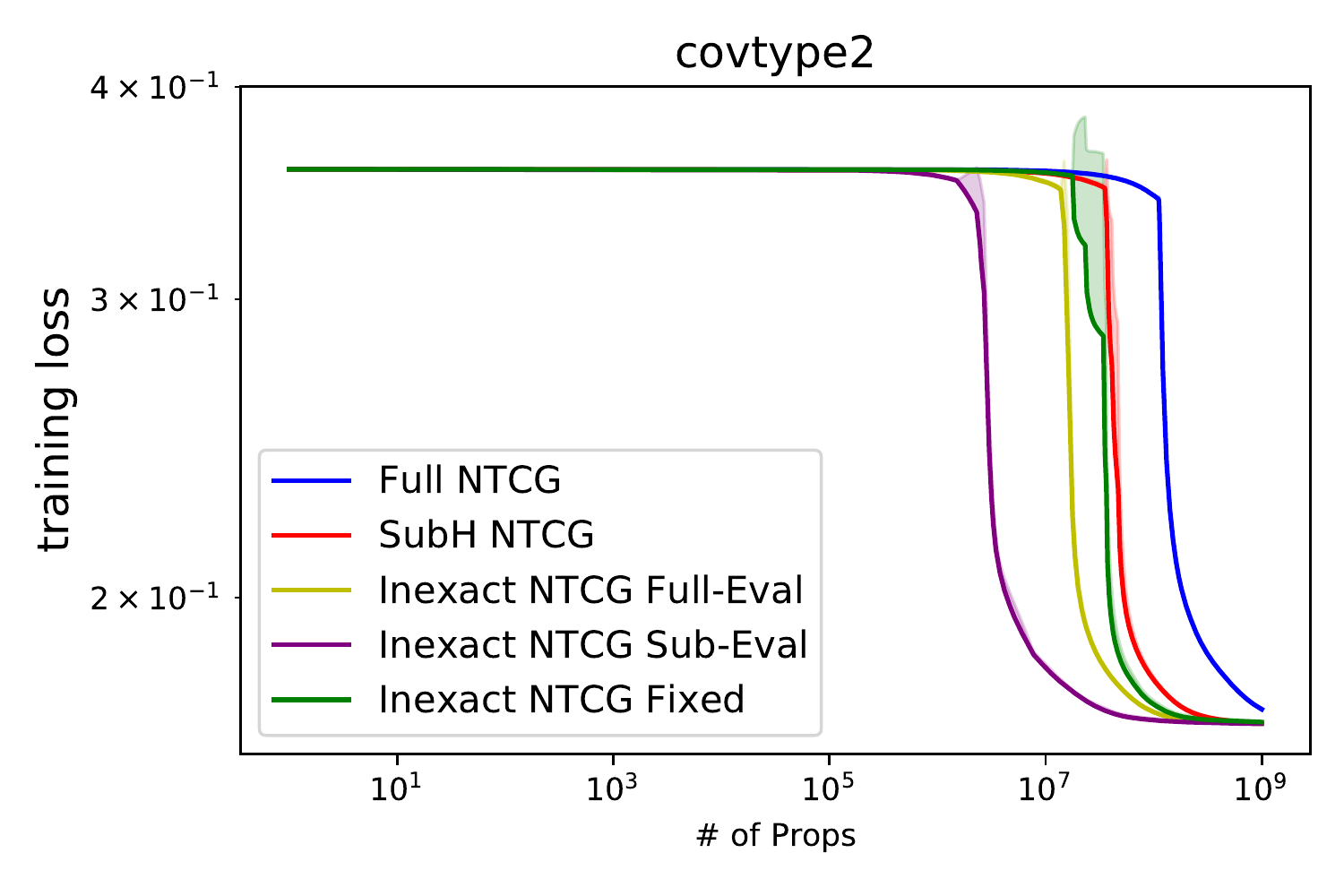}
\end{center}
\caption{
  Comparison between all variants of \texttt{NTCG} on \texttt{ijcnn1} and \texttt{covertype} datasets. 
}
\label{fig:nls_result_ntcg}
\end{figure}
%%%%%%%%%%%%%%%%%%%%%%%%%%%%%%%%%%%%%%%%%%%%%%%%%%%%%%%%%%%%%%%%%%%%%%%%%%%%%%%%%%%

\subsection{Multilayer perceptron}
\label{sec:mlp}
Here, we consider a slightly more complex setting than simple NLS and evaluate the performance of \cref{alg:inexact_withlinesearch,alg:fixed_stepsize} on several MLPs in the context of the image classification problem. 
For our experiments here, we will make use of the \texttt{MNIST} dataset, which is also available from \texttt{LIBSVM} library \citep{libsvm}. 
We consider three MLPs with one hidden layer, involving $ 16 $, $ 128 $, and $ 1024 $ neurons, respectively. All MLPs contain one output layer to determine the assigned class of the input image. 
The intermediate activation is chosen as the SoftPlus function~\citep{glorot2011deep}, which amounts to a smooth optimization problem. 
Table~\ref{tab:mlp} summarizes the total dimensions, in terms of $ n $ and $ d $, of the resulting optimization problems.  
\begin{table}[!htb]
	\centering
	\caption{The problem size for various MLPs.
		\label{tab:mlp}}
	\centering
	\begin{tabular}{lcccc}
		\toprule
		\multicolumn{1}{m{2cm}}{\centering Hidden Layer Size} & $n$ & $ d $  \\ 
		\midrule
		16 & 60,000 & 12,704 \\
		128 & 60,000 & 101,632 \\
		1,024 & 60,000 & 813,056 \\
		\bottomrule
	\end{tabular}
\end{table}

\cref{fig:mlp_result_ntcg} depicts the performance of all variants of NTCG that we consider in this paper. 
As can be seen, for all cases, our \texttt{Inexact NTCG Full-Eval} and \texttt{Inexact NTCG Sub-Eval} have the fastest convergence rate and achieve lower training loss as compared to alternatives. 

\begin{figure}[tbp]
\begin{center}
  \includegraphics[width=.325\textwidth]{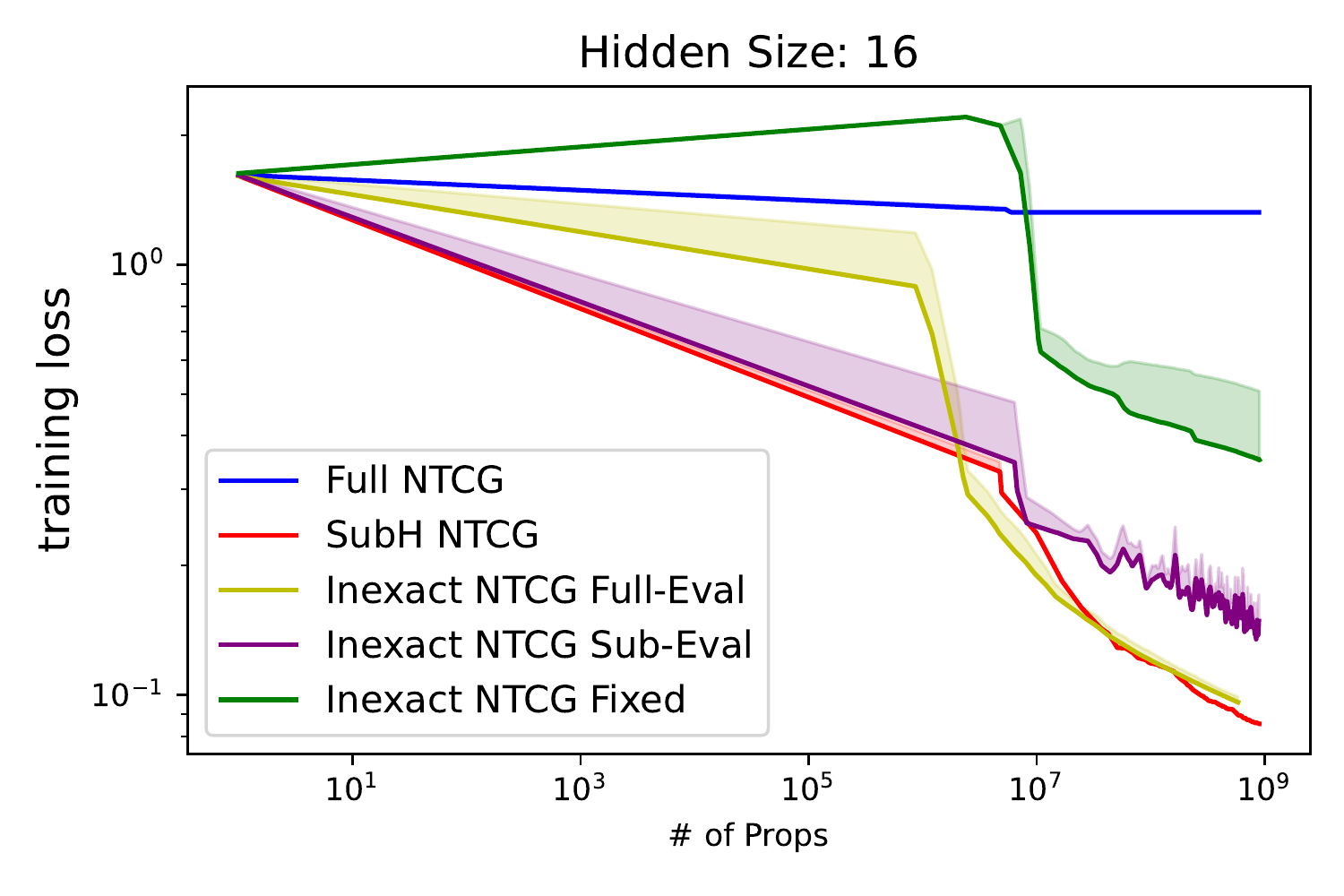}
  \includegraphics[width=.325\textwidth]{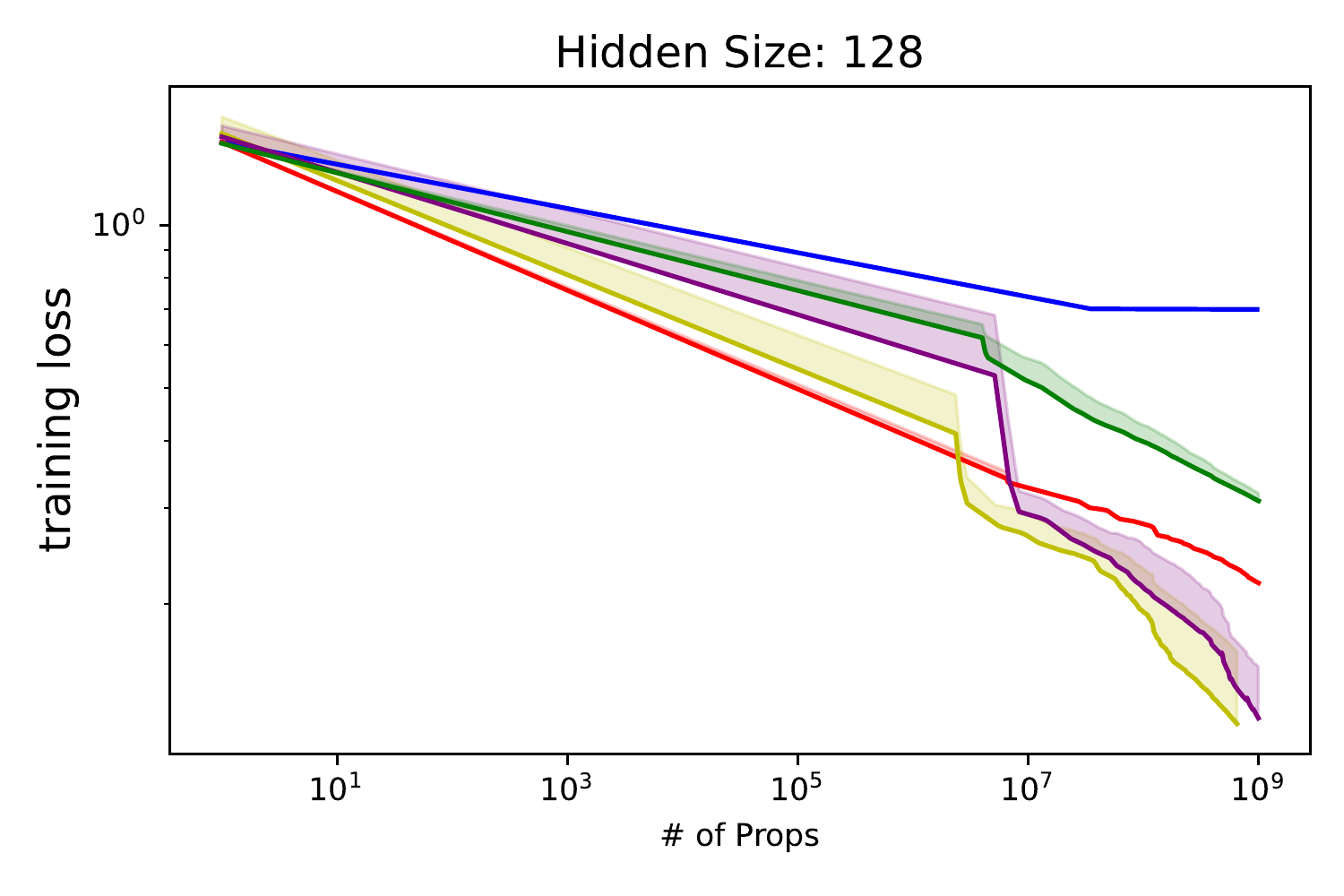}
  \includegraphics[width=.325\textwidth]{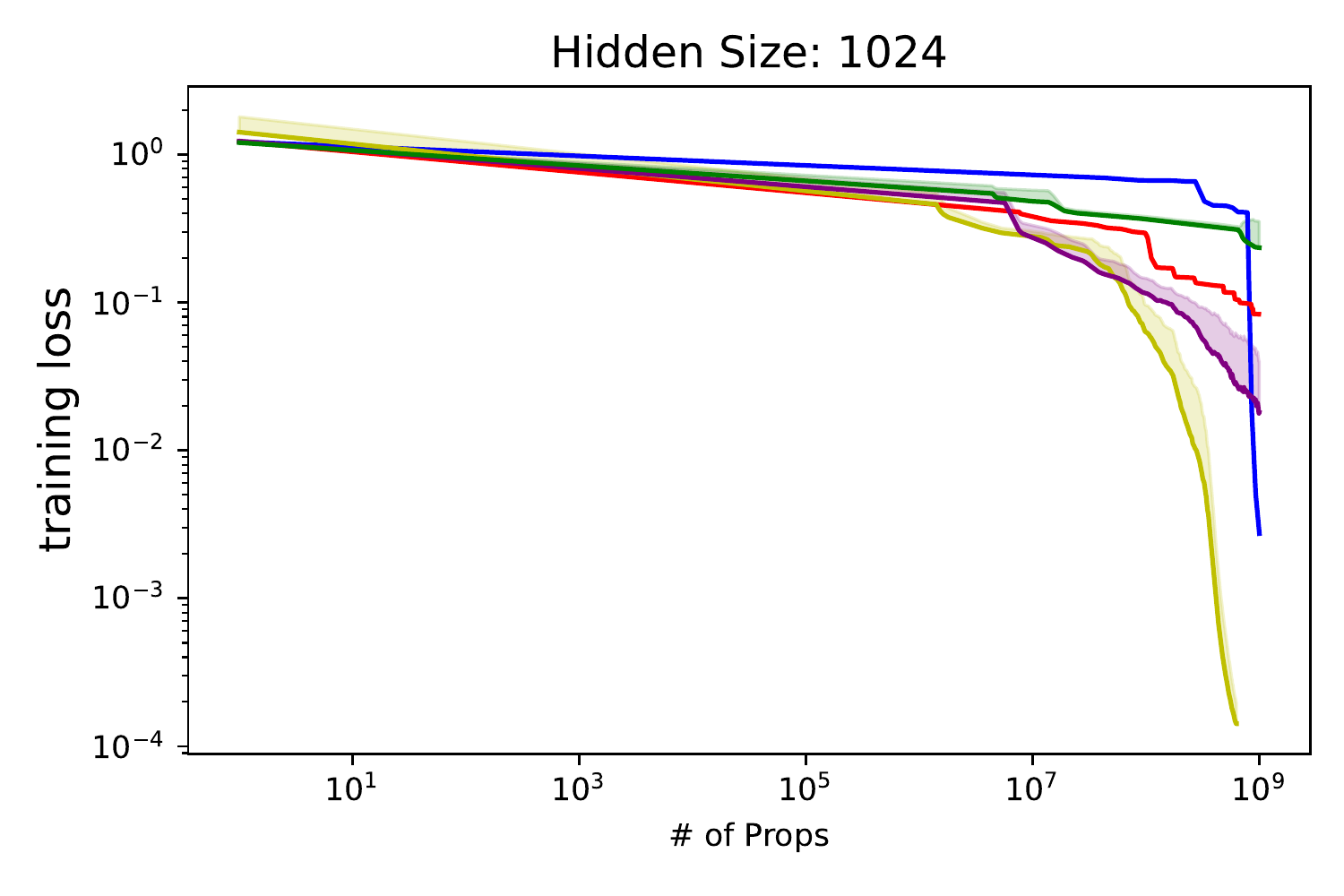}
\end{center}
\caption{
  Comparison between all variants of \texttt{NTCG} on several MLPs with different hidden-layer sizes: 16 (left), 128 (middle), and 1024 (right). 
}
\label{fig:mlp_result_ntcg}
\end{figure}

\subsection{Variational autoencoder}
\label{sec:vae}
We now evaluate the performance of 
%\ref{alg:inexact_withlinesearch} and~\ref{alg:fixed_stepsize}
\cref{alg:inexact_withlinesearch,alg:fixed_stepsize}
using a more complex setting of variational autoencoder (VAE) model.
Our VAE model consists of six fully-connect layers, which are structured as $784\rightarrow512\rightarrow256\rightarrow2\rightarrow256\rightarrow512\rightarrow784$. 
The intermediate activation and the output truncation functions, are respectively chosen as SoftPlus \citep{glorot2011deep} and Sigmoid \citep{glorot2011deep}. We again consider the \texttt{MNIST} dataset.

The results are shown in Figure~\ref{fig:vae}. 
Although we did not fine-tune the fixed step-sizes used within \texttt{Inexact NTCG Fixed} (as evidenced by its clear non-monotonic behavior), one can see that \texttt{Inexact NTCG Fixed} exhibits competitive performance. 
Again, as observed previously, \texttt{Inexact NTCG Full-Eval} and \texttt{Inexact NTCG Sub-Eval} have the fastest convergence rate among all of the variants.

\begin{figure}[tbp]
\begin{center}
  \includegraphics[width=.45\textwidth]{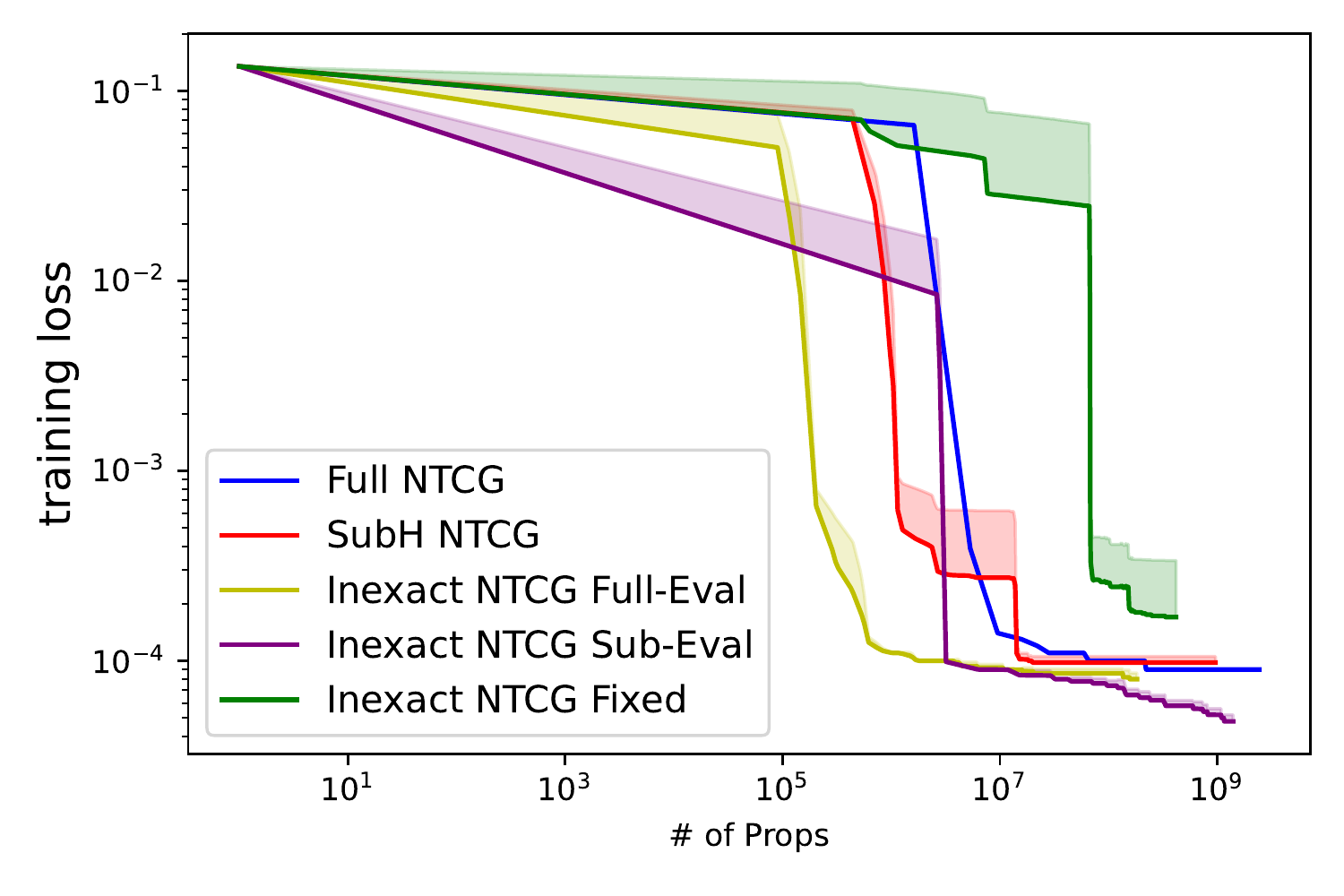}
  \end{center}
\caption{
  Comparison between all variants of \texttt{NTCG} on VAE. 
}
\label{fig:vae}
\end{figure}
\section{Conclusion}

We have considered inexact variants of the Newton-CG algorithm in which \refboth{approximations of gradient and Hessian are used.} 
\refone{\cref{alg:inexact_withlinesearch} employs approximations to the gradient and Hessian matrix at each step, and this inexact information is used to obtain an approximate Newton direction in Procedure \ref{alg:capped_cg}. 
However, to obtain the step-size, \cref{alg:inexact_withlinesearch} requires exact function values.
This issue is partially addressed in \cref{alg:fixed_stepsize}, where fixed step-sizes replace the line search. 
The drawbacks of the latter approach are that the fixed step-sizes are conservative and that they depend on some problem-dependent quantities that are generally unavailable, though known for  some important classes of machine learning problems. 
An ``ideal'' algorithm would allow for line searches using inexact function evaluations. 
One might be able to derive such a version using some further assumptions on the inexact function and the inexact gradient, such as those considered in \cite{paquette2020stochastic}, and by introducing randomness into the algorithm and the use of concentration bounds in the analysis. 
We intend to investigate these topics in future research.}

% \reftwo{Perhaps, one of the main drawbacks of \cref{alg:inexact_withlinesearch} is the need to evaluate the gradient exactly for computing the quantity $ \d^{T} \nabla f(\x_k) $, which is in tern needed to adjust the direction of the negative curvature vector. One possibility that can be explored in the future is to remove this quantity, and instead perform the line-search using both $\d$ and $-\d$ to determine the appropriate sign. The main difficulty here lies in knowing when to stop the backtracking line-search on either of these two directions. This is because one of these two direction will certainly be a direction of ascend. As a result, irrespective of the magnitude of the step-size, the backtracking operation never terminates. Identifying when this arises is key in providing analysis for this approach. Nonetheless, the cost of evaluating the gradient for most problems is roughly twice that of the function evaluation, e.g., , using automatic differentiation techniques \cite{baydin2018automatic,griewank1989automatic}. Hence, the computational savings from this approach might not be quite clear at this stage. However, there is undoubtedly merit in considering this direction further.}

We are especially interested in problems in which the objective has a ``finite-sum'' form, so the approximated gradients and Hessians are obtained by sampling randomly from the sum.
For all of our proposed variants, we showed that the iteration complexities needed to achieve approximate second-order criticality are essentially the same as that of the exact variants.
In particular, a variant that uses a fixed step size, rather than a step chosen adaptively by a backtracking line search, attains the same order of complexity as the other variants, despite never needing to evaluate the function itself. 

\refone{The dependence of our algorithms on Procedure \ref{alg:minimum_eigenvalue} implies the probabilistic nature of our results, which can be shown to hold with high-probability over the run of the algorithm.}

We demonstrate the advantages and shortcomings of the approach, in comparison with other methods, using several test problems.

\section*{Acknowledgements}
Fred Roosta was partially supported by the Australian Research Council through a Discovery Early Career Researcher Award (DE180100923).
Stephen Wright was partially supported by NSF Awards 1740707 and 2023239; DOE ASCR under Subcontract 8F-30039 from Argonne National Laboratory; Award N660011824020 from the DARPA Lagrange Program; and AFOSR under subcontract UTA20-001224 from the University of Texas-Austin. Michael Mahoney would also like to acknowledge DARPA, NSF, and ONR for providing partial support of this work.

\bibliographystyle{abbrvnat}
\bibliography{refs}

\begin{thebibliography}{46}
\providecommand{\natexlab}[1]{#1}
\providecommand{\url}[1]{\texttt{#1}}
\expandafter\ifx\csname urlstyle\endcsname\relax
  \providecommand{\doi}[1]{doi: #1}\else
  \providecommand{\doi}{doi: \begingroup \urlstyle{rm}\Url}\fi

\bibitem[Beck(2017)]{beck2017first}
A.~Beck.
\newblock \emph{{First-Order Methods in Optimization}}.
\newblock MOS-SIAM Series on Optimization. Society for Industrial and Applied
  Mathematics, 2017.
\newblock ISBN 9781611974997.

\bibitem[Bellavia and Gurioli(2021)]{bellavia2021stochastic}
S.~Bellavia and G.~Gurioli.
\newblock Stochastic analysis of an adaptive cubic regularization method under
  inexact gradient evaluations and dynamic hessian accuracy.
\newblock \emph{Optimization}, pages 1--35, 2021.
\newblock doi: 10.1080/02331934.2021.1892104.

\bibitem[Bellavia et~al.(2019)Bellavia, Gurioli, Morini, and
  Toint]{bellavia2019adaptive}
S.~Bellavia, G.~Gurioli, B.~Morini, and P.~L. Toint.
\newblock Adaptive regularization algorithms with inexact evaluations for
  nonconvex optimization.
\newblock \emph{SIAM Journal on Optimization}, 29\penalty0 (4):\penalty0
  2881--2915, 2019.

\bibitem[Bertsekas(1999)]{bertsekas1999nonlinear}
D.~P. Bertsekas.
\newblock \emph{Nonlinear programming}.
\newblock Athena scientific, 1999.

\bibitem[Blanchet et~al.(2019)Blanchet, Cartis, Menickelly, and
  Scheinberg]{blanchet2019convergence}
J.~Blanchet, C.~Cartis, M.~Menickelly, and K.~Scheinberg.
\newblock {Convergence Rate Analysis of a Stochastic Trust-Region Method via
  Supermartingales}.
\newblock \emph{INFORMS journal on optimization}, 1\penalty0 (2):\penalty0
  92--119, 2019.

\bibitem[Cartis and Scheinberg(2018)]{cartis2018global}
C.~Cartis and K.~Scheinberg.
\newblock Global convergence rate analysis of unconstrained optimization
  methods based on probabilistic models.
\newblock \emph{Mathematical Programming}, 169\penalty0 (2):\penalty0 337--375,
  2018.

\bibitem[Cartis et~al.(2011{\natexlab{a}})Cartis, Gould, and
  Toint]{cartis2011adaptiveI}
C.~Cartis, N.~I.~M. Gould, and P.~L. Toint.
\newblock {Adaptive cubic regularisation methods for unconstrained
  optimization. Part I: motivation, convergence and numerical results}.
\newblock \emph{Mathematical Programming}, 127\penalty0 (2):\penalty0 245--295,
  2011{\natexlab{a}}.

\bibitem[Cartis et~al.(2011{\natexlab{b}})Cartis, Gould, and
  Toint]{cartis2011adaptiveII}
C.~Cartis, N.~I.~M. Gould, and P.~L. Toint.
\newblock {Adaptive cubic regularisation methods for unconstrained
  optimization. Part II: worst-case function-and derivative-evaluation
  complexity}.
\newblock \emph{Mathematical programming}, 130\penalty0 (2):\penalty0 295--319,
  2011{\natexlab{b}}.

\bibitem[Cartis et~al.(2012)Cartis, Gould, and Toint]{cartis2012complexity}
C.~Cartis, N.~I.~M. Gould, and P.~L. Toint.
\newblock {Complexity bounds for second-order optimality in unconstrained
  optimization}.
\newblock \emph{Journal of Complexity}, 28\penalty0 (1):\penalty0 93--108,
  2012.

\bibitem[Chang and Lin(2011)]{libsvm}
C.-C. Chang and C.-J. Lin.
\newblock {LIBSVM}: A library for support vector machines.
\newblock \emph{ACM Transactions on Intelligent Systems and Technology},
  2:\penalty0 27:1--27:27, 2011.

\bibitem[Choromanska et~al.(2015)Choromanska, Henaff, Mathieu, Arous, and
  LeCun]{choromanska2015loss}
A.~Choromanska, M.~Henaff, M.~Mathieu, G.~B. Arous, and Y.~LeCun.
\newblock The loss surfaces of multilayer networks.
\newblock In \emph{Artificial intelligence and statistics}, pages 192--204.
  PMLR, 2015.

\bibitem[Conn et~al.(2000)Conn, Gould, and Toint]{conn2000trust}
A.~R. Conn, N.~I. Gould, and P.~L. Toint.
\newblock \emph{{Trust region methods}}.
\newblock SIAM, 2000.

\bibitem[Curtis et~al.(2014)Curtis, Robinson, and Samadi]{curtis2014trust}
F.~E. Curtis, D.~P. Robinson, and M.~Samadi.
\newblock A trust region algorithm with a worst-case iteration complexity of $
  \mathcal{O}(\epsilon^{-3/2}) $for nonconvex optimization.
\newblock \emph{COR@ L Technical Report 14T-009, Lehigh University,, Bethlehem,
  PA, USA}, 2014.

\bibitem[Curtis et~al.(2021)Curtis, Robinson, Royer, and Wright]{Cur19a}
F.~E. Curtis, D.~P. Robinson, C.~W. Royer, and S.~J. Wright.
\newblock Trust-region {Newton-CG} with strong second-order complexity
  guarantees for nonconvex optimization.
\newblock \emph{{SIAM} Journal on Optimization}, 31:\penalty0 518--544, 2021.

\bibitem[Dauphin et~al.(2014)Dauphin, Pascanu, Gulcehre, Cho, Ganguli, and
  Bengio]{dauphin2014identifying}
Y.~N. Dauphin, R.~Pascanu, C.~Gulcehre, K.~Cho, S.~Ganguli, and Y.~Bengio.
\newblock {Identifying and attacking the saddle point problem in
  high-dimensional non-convex optimization}.
\newblock In \emph{Advances in neural information processing systems}, pages
  2933--2941, 2014.

\bibitem[Duchi et~al.(2011)Duchi, Hazan, and Singer]{duchi2011adaptive}
J.~Duchi, E.~Hazan, and Y.~Singer.
\newblock Adaptive subgradient methods for online learning and stochastic
  optimization.
\newblock \emph{Journal of Machine Learning Research}, 12\penalty0
  (Jul):\penalty0 2121--2159, 2011.

\bibitem[Ge et~al.(2015)Ge, Huang, Jin, and Yuan]{ge2015escaping}
R.~Ge, F.~Huang, C.~Jin, and Y.~Yuan.
\newblock Escaping from saddle points-online stochastic gradient for tensor
  decomposition.
\newblock In \emph{Proceedings of The 28th Conference on Learning Theory},
  volume~40, pages 797--842. PMLR, 2015.

\bibitem[Glorot et~al.(2011)Glorot, Bordes, and Bengio]{glorot2011deep}
X.~Glorot, A.~Bordes, and Y.~Bengio.
\newblock Deep sparse rectifier neural networks.
\newblock In \emph{Proceedings of the fourteenth international conference on
  artificial intelligence and statistics}, volume~15, pages 315--323. PMLR,
  2011.

\bibitem[Gratton et~al.(2018)Gratton, Royer, Vicente, and
  Zhang]{gratton2018complexity}
S.~Gratton, C.~W. Royer, L.~N. Vicente, and Z.~Zhang.
\newblock Complexity and global rates of trust-region methods based on
  probabilistic models.
\newblock \emph{IMA Journal of Numerical Analysis}, 38\penalty0 (3):\penalty0
  1579--1597, 2018.

\bibitem[Hillar and Lim(2013)]{hillar2013most}
C.~J. Hillar and L.-H. Lim.
\newblock {Most tensor problems are NP-hard}.
\newblock \emph{Journal of the ACM (JACM)}, 60\penalty0 (6):\penalty0 45, 2013.

\bibitem[Jin et~al.(2017)Jin, Ge, Netrapalli, Kakade, and
  Jordan]{jin2017escape}
C.~Jin, R.~Ge, P.~Netrapalli, S.~M. Kakade, and M.~I. Jordan.
\newblock How to escape saddle points efficiently.
\newblock In \emph{Proceedings of the 34th International Conference on Machine
  Learning-Volume 70}, volume~70, pages 1724--1732. PMLR, 2017.

\bibitem[Kingma and Ba(2014)]{kingma2014adam}
D.~P. Kingma and J.~Ba.
\newblock Adam: A method for stochastic optimization.
\newblock \emph{arXiv preprint arXiv:1412.6980}, 2014.

\bibitem[Lan(2020)]{lan2020first}
G.~Lan.
\newblock \emph{{First-order and Stochastic Optimization Methods for Machine
  Learning}}.
\newblock Springer Series in the Data Sciences. Springer International
  Publishing, 2020.
\newblock ISBN 9783030395674.

\bibitem[LeCun et~al.(2012)LeCun, Bottou, Orr, and
  M{\"u}ller]{lecun2012efficient}
Y.~A. LeCun, L.~Bottou, G.~B. Orr, and K.-R. M{\"u}ller.
\newblock {Efficient backprop}.
\newblock In \emph{Neural networks: Tricks of the trade}, pages 9--48.
  Springer, 2012.

\bibitem[Levy(2016)]{levy2016power}
K.~Y. Levy.
\newblock {The Power of Normalization: Faster Evasion of Saddle Points}.
\newblock \emph{arXiv preprint arXiv:1611.04831}, 2016.

\bibitem[Lin et~al.(2020)Lin, Li, and Fang]{lin2020accelerated}
Z.~Lin, H.~Li, and C.~Fang.
\newblock \emph{{Accelerated Optimization for Machine Learning: First-Order
  Algorithms}}.
\newblock Springer Singapore, 2020.
\newblock ISBN 9789811529108.

\bibitem[Liu and Roosta(2021)]{liu2019stability}
Y.~Liu and F.~Roosta.
\newblock {Convergence of Newton-MR under Inexact Hessian Information}.
\newblock \emph{SIAM Journal on Optimization}, 31\penalty0 (1):\penalty0
  59--90, 2021.

\bibitem[Mishra and Giorgi(2008)]{mishra2008invexity}
S.~K. Mishra and G.~Giorgi.
\newblock \emph{{Invexity and Optimization}}, volume~88.
\newblock Springer Science \& Business Media, 2008.

\bibitem[Murty and Kabadi(1987)]{murty1987some}
K.~G. Murty and S.~N. Kabadi.
\newblock {Some NP-complete problems in quadratic and nonlinear programming}.
\newblock \emph{Mathematical Programming}, 39\penalty0 (2):\penalty0 117--129,
  1987.

\bibitem[Nesterov and Polyak(2006)]{nesterov2006cubic}
Y.~Nesterov and B.~T. Polyak.
\newblock {Cubic regularization of Newton method and its global performance}.
\newblock \emph{Mathematical Programming}, 108\penalty0 (1):\penalty0 177--205,
  2006.

\bibitem[Nocedal and Wright(2006)]{numopt}
J.~Nocedal and S.~J. Wright.
\newblock \emph{{Numerical Optimization}}.
\newblock Springer Science \& Business Media, second edition, 2006.

\bibitem[Paquette and Scheinberg(2020)]{paquette2020stochastic}
C.~Paquette and K.~Scheinberg.
\newblock A stochastic line search method with expected complexity analysis.
\newblock \emph{SIAM Journal on Optimization}, 30\penalty0 (1):\penalty0
  349--376, 2020.

\bibitem[Roosta and Mahoney(2019)]{roosta2019sub}
F.~Roosta and M.~W. Mahoney.
\newblock {Sub-sampled Newton methods}.
\newblock \emph{Mathematical Programming}, 174\penalty0 (1-2):\penalty0
  293--326, 2019.

\bibitem[Roosta et~al.(2018)Roosta, Liu, Xu, and Mahoney]{roosta2018newton}
F.~Roosta, Y.~Liu, P.~Xu, and M.~W. Mahoney.
\newblock {Newton-MR: Newton's Method Without Smoothness or Convexity}.
\newblock \emph{arXiv preprint arXiv:1810.00303}, 2018.

\bibitem[Royer and Wright(2018)]{royer2018complexity}
C.~W. Royer and S.~J. Wright.
\newblock Complexity analysis of second-order line-search algorithms for smooth
  nonconvex optimization.
\newblock \emph{SIAM Journal on Optimization}, 28\penalty0 (2):\penalty0
  1448--1477, 2018.

\bibitem[Royer et~al.(2020)Royer, O'Neill, and Wright]{royer2018newton}
C.~W. Royer, M.~O'Neill, and S.~J. Wright.
\newblock {A Newton-CG Algorithm with Complexity Guarantees for Smooth
  Unconstrained Optimization}.
\newblock \emph{Mathematical Programming, Series A}, 180:\penalty0 451--488,
  2020.

\bibitem[Saxe et~al.(2013)Saxe, McClelland, and Ganguli]{saxe2013exact}
A.~M. Saxe, J.~L. McClelland, and S.~Ganguli.
\newblock {Exact solutions to the nonlinear dynamics of learning in deep linear
  neural networks}.
\newblock \emph{arXiv preprint arXiv:1312.6120}, 2013.

\bibitem[Shewchuk(1994)]{shewchuk1994introduction}
J.~R. Shewchuk.
\newblock An introduction to the conjugate gradient method without the
  agonizing pain.
\newblock 1994.

\bibitem[Steihaug(1983)]{steihaug1983conjugate}
T.~Steihaug.
\newblock {The conjugate gradient method and trust regions in large scale
  optimization}.
\newblock \emph{SIAM Journal on Numerical Analysis}, 20\penalty0 (3):\penalty0
  626--637, 1983.

\bibitem[Tripuraneni et~al.(2018)Tripuraneni, Stern, Jin, Regier, and
  Jordan]{tripuraneni2017stochasticcubic}
N.~Tripuraneni, M.~Stern, C.~Jin, J.~Regier, and M.~I. Jordan.
\newblock {Stochastic cubic regularization for fast nonconvex optimization}.
\newblock In \emph{Advances in neural information processing systems}, pages
  2899--2908. PMLR, 2018.

\bibitem[Wright and Recht(2021)]{WriR21}
S.~J. Wright and B.~Recht.
\newblock \emph{Optimization for Data Analysis}.
\newblock Cambridge University Press, 2021.
\newblock (To appear.).

\bibitem[Xie and Wright(2021)]{XieW21a}
Y.~Xie and S.~J. Wright.
\newblock Complexity of projected {Newton} methods in bound-constrained
  optimization.
\newblock {Technical Report} arXiv:2103:15989, University of Wisconsin-Madison,
  March 2021.
\newblock In preparation.

\bibitem[Xu et~al.(2020{\natexlab{a}})Xu, Roosta, and
  Mahoney]{xuNonconvexEmpirical2017}
P.~Xu, F.~Roosta, and M.~W. Mahoney.
\newblock {Second-order optimization for non-convex machine learning: An
  empirical study}.
\newblock In \emph{Proceedings of the 2020 SIAM International Conference on
  Data Mining}, pages 199--207. SIAM, 2020{\natexlab{a}}.

\bibitem[Xu et~al.(2020{\natexlab{b}})Xu, Roosta, and
  Mahoney]{xuNonconvexTheoretical2017}
P.~Xu, F.~Roosta, and M.~W. Mahoney.
\newblock {Newton-type methods for non-convex optimization under inexact
  Hessian information}.
\newblock \emph{Mathematical Programming}, 184\penalty0 (1):\penalty0 35--70,
  2020{\natexlab{b}}.

\bibitem[Yao et~al.(2020)Yao, Xu, Roosta, and Mahoney]{yao2018inexact}
Z.~Yao, P.~Xu, F.~Roosta, and M.~W. Mahoney.
\newblock {Inexact non-convex Newton-type methods}.
\newblock \emph{INFORMS Journal on Optimization}, 2020.
\newblock doi.org/10.1287/ijoo.2019.0043.

\bibitem[Zhang et~al.(2019)Zhang, Mei, Shi, and Xu]{zhang2019robustness}
R.~Zhang, Y.~Mei, J.~Shi, and H.~Xu.
\newblock Robustness and tractability for non-convex m-estimators.
\newblock \emph{arXiv preprint arXiv:1906.02272}, 2019.

\end{thebibliography}

\end{document}